\documentclass[aop, preprint]{imsart}

\usepackage{amsmath,amssymb,amsfonts,amsthm,xspace}
\usepackage{graphicx}
\usepackage{geometry}
\usepackage{amsthm}
\usepackage[linktocpage=true]{hyperref}
\usepackage[usenames,dvipsnames]{color}

\arxiv{arXiv:1211.6974}

\startlocaldefs

\newtheorem*{theorem*}{Theorem}
\newtheorem{theorem}{Theorem}
\newtheorem{prop}[theorem]{Proposition}
\newtheorem{lemma}[theorem]{Lemma}
\newtheorem{corollary}[theorem]{Corollary}
\newtheorem{conjecture}[theorem]{Conjecture}
\newtheorem{remark}[theorem]{Remark}
\newcommand{\be}{\begin{equation}}
\newcommand{\ee}{\end{equation}}

\newcommand{\R}{{\mathbb R}}

\newcommand{\M}{{\mathcal M}}
\newcommand{\C}{{\mathbb C}}
\newcommand{\T}{{\mathbb T}}
\newcommand{\G}{{\mathcal G}}
\newcommand{\SL}{\mathrm{SL}_2({\mathbb C})}
\newcommand{\SU}{\mathrm{SU}(2)}

\newcommand{\Tr}{\text{Tr}}

\newcommand{\old}[1]{}

\newcommand{\A}{\mathcal A}

\renewcommand{\P}{\mathbb{P}}
\newcommand{\E}{\mathbb{E}}

\newcommand{\B}{{\mathcal B}}
\newcommand{\W}{{\mathcal W}}

\newcommand{\eps}{{\varepsilon}}
\newcommand{\ga}{{\bf \Gamma}}
\newcommand{\SLE}{\mathrm{SLE}}

\endlocaldefs

\begin{document}

\begin{frontmatter}

\title{Random curves on surfaces induced from the Laplacian determinant}
\runtitle{Random curves on surfaces induced from the Laplacian determinant}

\author{\fnms{Adrien} \snm{Kassel}\corref{}\ead[label=e1]{adrien.kassel@math.ethz.ch}\thanksref{t1}}
\and
\author{\fnms{Richard} \snm{Kenyon}\ead[label=e2]{rkenyon@math.brown.edu}\thanksref{t2}}

\thankstext{t1}{Partially supported by Fondation Sciences Math\'ematiques de Paris. Most of this work was completed while affiliated with ENS Paris.} 
\thankstext{t2}{Research supported by  NSF grant DMS-1208191 and the Simons Foundation.}

\runauthor{A. Kassel and R. Kenyon}


\address{Adrien Kassel\\ Departement Mathematik\\ETH\\R\"{a}mistrasse 101\\8092 Z\"{u}rich\\Switzerland\\ \printead{e1}}

\address{Richard Kenyon\\Mathematics Department\\Brown University\\151 Thayer St.\\Providence, Rhode Island 02912\\USA\\ \printead{e2}}

\begin{abstract}
We define natural probability measures
on finite multicurves (finite collections of 
pairwise disjoint simple closed curves) on curved surfaces. 
These measures arise as universal scaling limits of probability measures on cycle-rooted spanning forests (CRSFs) 
on graphs embedded on a surface with a Riemannian metric, in the limit as the mesh size tends to zero.
These in turn are defined from the Laplacian
determinant and depend on the choice of a unitary connection on the surface.  

Wilson's algorithm for generating spanning trees on a graph
generalizes to a cycle-popping algorithm for generating CRSFs for a general family of weights on the cycles.
We use this to sample the above measures.
The sampling algorithm, which relates these measures to the loop-erased random walk, is also used to prove 
tightness of the sequence of measures, a key step in the proof of their convergence.

We set the framework for the study of these probability measures and their scaling limits and state some of their properties.
\end{abstract}



\end{frontmatter}


\tableofcontents

\section{Introduction}

Classical statistical mechanics deals with systems of large numbers of particles interacting through
local forces. These systems are naturally defined on Euclidean spaces, so that the notions of scaling limit and scale 
invariance make sense.
In this work we define a large family of statistical mechanical systems on curved spaces: 
curved surfaces with possibly nontrivial topology,
where the curvature and topology play both a local and global role in the underlying probability measure.
By scaling limit in such a context we mean that as the system size grows we shrink
the discretization parameter so that the metric properties
of the underlying surface remain constant. 

A fundamental property of our scaling limits is that they are \emph{universal}:
they are independent of the details of the discrete approximating sequence.
In other words they are natural, parameter-free systems on the curved surface itself, depending only on its
geometry and topology.  

Let us detail the systems we consider. A \emph{cycle-rooted spanning forest} (CRSF) on a graph~$\G$ is
a subgraph each of whose components contains a unique cycle, or equivalently,
contains as many vertices as edges, see Figure~\ref{crsf}. A \emph{cycle-rooted spanning tree} (CRST) is a connected CRSF. 

Natural probability measures on CRSFs
arising from the determinant of the graph Laplacian were introduced in \cite{Ke1}: the probability of a CRSF
is proportional to the product over its cycles
of a certain function of the cycle, depending on 
the holonomy of a discrete $\C^*$- or $\SL$-connection.
The interest of these measures is that they can give to a cycle a weight which is a function of its shape. Furthermore, these measures are determinantal viewed as point processes on the set of edges.

\begin{figure}[ht]\label{crsf}
\centering
\includegraphics[width=10cm]{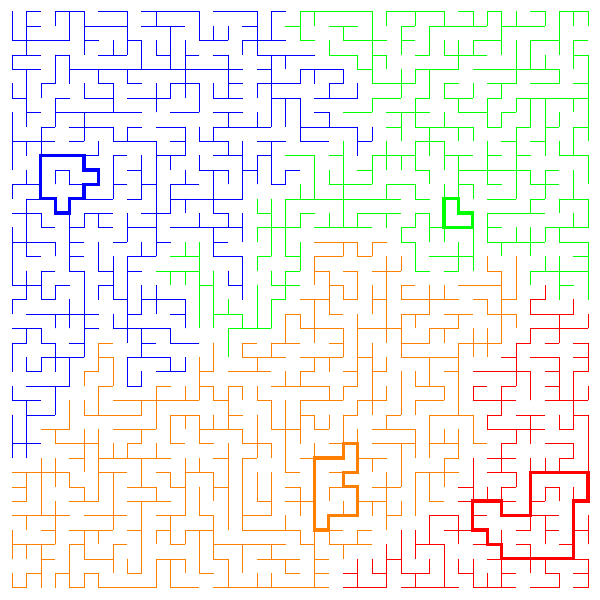}
\caption{A CRSF on the $40\times 40$ square grid; each connected component is in a different color and the cycles are in bold.}
\end{figure}

We study here the \emph{scaling limits} of these measures: their limits on a sequence of finer and finer
graphs approximating a surface with a fixed Riemannian metric.
By \emph{surface} we will mean here an oriented smooth surface with a Riemannian metric. By \emph{approximation} of a surface $\Sigma$ we will mean a sequence of graphs $(\G_n)$ geodesically embedded on $\Sigma$ and
\emph{conformally approximating} it in a sense defined below
(essentially, the simple random walk on $\G_n$ converges to Brownian motion on $\Sigma$). We endow the space of multiloops on $\Sigma$ with a natural topology and show the weak convergence of the above probability measures on multiloops on $\G_n.$

We may informally summarize our main statement as follows.
\begin{theorem*}[see Theorem~\ref{main}]
For any Riemannian surface, there is a natural probability measure on finite collections of disjoint simple closed curves drawn on it (which are fractal, locally resembling the Schramm-Loewner evolution $\SLE_2$). This measure is the universal scaling limit of natural discrete measures on CRSF loops defined on graphs conformally approximating the surface. When the surface has nontrivial topology and its metric is flat, these curves are noncontractible. In the case of the flat Euclidean disk, the probability measure is degenerate but a measure on single curves is obtained by a limiting procedure.
\end{theorem*}

We study scaling limits in two different settings: a topological setting and a geometrical setting.

In the topological setting, we consider only the conformal class of the metric on $\Sigma$.
Define a \emph{noncontractible} CRSF to be a CRSF with no contractible cycles.
Let $\left(\G_n\right)$ be an approximating sequence and $\mu^n_{nonc}$ be the uniform measure on noncontractible CRSFs of $\G_n$.
We show (Theorem~\ref{noncontractible}) that the cycle process~$\P_{nonc}^n$ of the 
$\mu^n_{nonc}$-random CRSF on $\G_n$
converges to a random loop process~$\P_{nonc}$ on $\Sigma$,
\emph{independent} of the approximating sequence $\G_n$. The limit only depends on the conformal class of $\Sigma$, in the following sense.
Let $z_1,\dots,z_k$ be distinct points of~$\Sigma$.
For any isotopy class of sets of pairwise disjoint simple loops 
$\{\gamma_1,\dots,\gamma_m\}$ of 
$\Sigma\setminus \{z_1,\dots,z_k\}$, the probability that
a random noncontractible CRSF on $\G_n$ has $m$ cycles, and these are 
isotopic to the $\gamma_i$, has a probability
converging as $n\to\infty$ to a limit independent of the approximating 
sequence $\G_n$. This refines the result of \cite{Ke2} who showed
that (for the dimer model, which is 
closely related to the CRSF model via Temperley's bijection \cite{KPW}) 
the distribution of the homotopy classes of the cycles in~$\pi_1(\Sigma)$ 
has a conformally invariant limit when~$\Sigma$ is a planar domain. A related (infinite) measure~$\mu_{-2}$ on simple loops on $\Sigma$ was recently
constructed in a very different manner by Benoist and Dubedat in \cite{BD}. The measure $\mu_{-2}$ restricted to noncontractible loops of the annulus (and normalized to be a probability measure) is the same as our measure~$\P_{nonc}$ conditioned to have one loop. When extended to all surfaces in such a way that a ``conformal restriction" property is satisfied, 
the measure $\mu_{-2}$ was conjectured to exist by Kontsevitch and Suhov. Further relations between our measures and $\mu_{-2}$ will be considered in a forthcoming paper.

In the geometrical setting, we take into account the metric on $\Sigma$, and in particular its curvature. 
Let $\left(\G_n\right)$ be a sequence of finite graphs conformally approximating $\Sigma$.
Associated to this data is a discrete connection on a complex 
line bundle over $\G_n$ arising from
the Levi-Civita connection on the tangent bundle on 
$\Sigma$. It is defined up to gauge equivalence by the property that the holonomy
around a loop is $e^{i\theta}$ where $\theta$ is the enclosed curvature.
From this connection $\Phi_{LC}^n$
we construct a natural probability measure $\mu^n_{LC}$ on CRSFs
on $\G_n$: each CRSF has a probability proportional to the product over its cycles of $2-2\cos\theta$, where $\theta$ is the curvature enclosed. The corresponding loop process $\P_{LC}^n$ is shown (Theorem~\ref{main}) to converge to a probability measure $\P_{LC}$ on multicurves on $\Sigma$, independent of the approximating sequence.

When the surface $\Sigma$ is contractible, we define another measure $\mu^n_{LC^0}$ which is in some sense
more natural. This measure is a limit  when $\varepsilon\to 0$ of the CRSF connection measures for the connection with curvature 
$e^{i \varepsilon \theta}$, a limit which was introduced in~\cite{Ke1}. 
This yields a measure on cycle-rooted spanning trees (CRSTs, that is, CRSFs with one component)
with weight proportional to $\theta^2$ where $\theta$ is the enclosed curvature. 
We show (Theorem~\ref{main}) that the loop measure $\P_{LC^0}^n$ converges to a measure $\P_{LC^0}$ on simple closed curves on~$\Sigma$, again independent of the approximating sequence.

We give a ``cycle-popping" algorithm (Theorem~\ref{algo}) for 
rapid exact sampling from the above measures (as well as more general measures), generalizing the well-known 
cycle-popping algorithm of Wilson \cite{Wi} for generating
uniform spanning trees. One simply runs Wilson's algorithm, and when a cycle
is created, flip a coin (with bias depending on the cycle weight) to decide whether to keep it or not.

We use this sampling algorithm to sample approximations of the above scaling limits.
In all three cases, the cycle-popping algorithm is an essential part of the convergence
argument: it is used to show tightness of the sequence of measures (Section~\ref{tightnesspart}).  

Another essential result shows that there are almost surely a finite number of components
in a random CRSF, and the scaling limits of the loops are nondegenerate, in the sense that they
do not shrink to points in the limit. This is accomplished by computing the universal limit of the probability of having no loops and by exploring in a Markovian way via the algorithm the surface with positive probability of creating macroscopic loops at each step. This implies a super-exponential tail for the number of loops which excludes the fact of having microscopic loops since otherwise their number would be infinite by the weak large scale dependence of the process.

We give samples from the measures $\mu_{LC}$ and $\mu_{LC^0}$ 
for the round sphere (Figure~\ref{LC1}), a saddle surface (Figure~\ref{LC3}, left) and a compact disk in the 
Poincar\'e plane (Figure~\ref{LC3}, right), and for the measure $\mu_{nonc}$ on a flat torus (Figure~\ref{inc1}) and planar domains (Figures~\ref{LaminationBranches}).
For $\mu_{LC}$ these are conditional samples,
conditioned on having only loops with area (curvature) bounded by $\pi/2$; our 
sampling algorithm does not work without this condition (see, however, \cite{HKPV} where it is shown how to sample from any determinantal process with Hermitian kernel, of which $\mu_{LC}$ is one).

The paper is organized as follows. In Section~\ref{algorithm} we introduce the sampling algorithm and prove its correctness. In Section~\ref{measures} we introduce the probability measures on CRSFs on graphs on surfaces and show how they are exactly sampled by the algorithm. In Section~\ref{scaling} we show that the probability measures on loops that these induce converge to probability measures on the space of multiloops of the surface (this section contains the proof of our main statement). Section~\ref{properties} enumerates some of the properties of the measures on loops considered in the paper. The paper concludes with a list of open questions in Section~\ref{questions}. 
\medskip

\noindent{\bf Acknowledgements.} We would like to thank Thierry L\'evy, David Wilson and Wei Wu for helpful discussions, and the referee for helpful suggestions.

\section{A general sampling algorithm}\label{algorithm}
 
An \emph{oriented} CRSF is a CRSF in which
each cycle has a chosen orientation. A measure on CRSFs induces a measure
on oriented CRSFs by giving each cycle an independent $1/2-1/2$-chosen orientation, and a measure on oriented CRSFs induces one on CRSFs by forgetting the orientation.

Let $\G=(V,E)$ be a finite graph with vertex and edge sets $V$ and $E$, respectively, and $c:E\to\R_{>0}$ a positive function on the edges
which we call the \emph{conductance}. 
Let $\alpha$ be a function which assigns to each oriented simple loop $\gamma$ in $\G$ a positive weight $\alpha(\gamma)\in[0,1]$. 
We allow loops $\gamma$ consisting of two edges (a backtrack: an edge which is immediately traversed in the reverse direction). These functions $c,\alpha$ define a probability measure $\mu=\mu_{c,\alpha}$ on oriented CRSFs, giving an oriented CRSF $\Gamma$ a probability proportional to 
$\prod_{e\in \Gamma}c(e)\prod_{\text{cycles }\gamma\subset\Gamma}\alpha(\gamma)$. We describe an algorithm to sample an oriented CRSF according to the measure $\mu$. 

We note that this sampling algorithm requires $\alpha\in[0,1]$; it will not
work without modification for larger $\alpha$. 
In the special case where $\alpha=1$ and $c=1$, the algorithm samples according to the uniform measure on oriented CRSFs.

Let us describe a cycle-popping procedure, named $P[w,\Gamma]$, which takes as arguments $w$ a vertex and 
$\Gamma$ an oriented subgraph of $\G$ not containing $w$, and outputs another oriented subgraph of $\G$ containing $\Gamma$ and $w$. The procedure is the following:
start at vertex $w$ and perform a simple random walk 
(with each step proportional to the conductances) until it first reaches a vertex $v$ which either belongs to $\Gamma$ or is the first self-intersection of its path; 
\begin{itemize}
\item If $v$ is in $\Gamma$, then replace $\Gamma$ by the union of $\Gamma$ and the oriented path just traced by the random walk.

\item If $v$ is the first self-intersection, let $\gamma$ denote the oriented cycle thus obtained, and sample a $\{0,1\}$-Bernoulli random variable with success probability $\alpha(\gamma)$; 
\begin{itemize}
\item If the outcome is $1$, then replace $\Gamma$ by the union of $\Gamma$ and the oriented path just traced by the random walk.
 
\item If the outcome is $0$, erase the cycle that was just closed and continue to perform the random walk from $v$ until it reaches $\Gamma$ or self-intersects, in which case repeat the above instructions.
\end{itemize}
\end{itemize}

The algorithm, called $\mathcal{A}$, is then the following: start with $\Gamma$ empty and $w$ an arbitrary vertex, and perform $P[w,\Gamma]$. 
If the output $\Gamma'$ is not a CRSF, take a new vertex $w'\not\in\Gamma'$ and perform $P[w',\Gamma']$, and so on 
until the output contains all vertices of~$\G$. Note that the output of $\A$ is an oriented CRSF and that we forget the information about the order of construction of the cycles.

\begin{theorem}\label{algo}
If $\alpha(\gamma)>0$ for some $\gamma$ then the algorithm $\mathcal{A}$ terminates and its output is an oriented CRSF, sampled according to the measure $\mu$. 
\end{theorem}

\begin{remark}
Note that if in the above definition of $\mathcal{A}$ one starts with $\Gamma$ equal to $S$, a distinguished set of vertices in $\G$, then the algorithm samples an oriented \emph{essential CRSF}, see definition in Section \ref{wiredbdry} below.
\end{remark}
 
\begin{proof}
Following the proof of Wilson's algorithm~\cite{Wi}, we construct an equivalent description, denoted $\A'$, of the algorithm $\A$.

Let us define $\A'$ in the following way. Consider, over each vertex $v\in V$, an infinite sequence $X^{(v)}$
of i.i.d random
variables $X^{(v)}=(X_1^{(v)},X_2^{(v)},\dots)$, 
each distributed as a random neighbor of $v$ according to the conductance measure, that is, for each $i\geq 1$ and each neighbor $w$ of $v$, we have
$$\P\left(X_i^{(v)}=w\right)=\frac{c(vw)}{\sum_{v'\sim v}c(vv')}\,.$$
We represent $X^{(v)}$ as an infinite stack of cards, with only $X_1^{(v)}$ 
being visible at the top of the stack. 

We draw an edge from each vertex $v$
to the neighbor shown on the top of the stack $X^{(v)}$;
the oriented graph thus seen is an oriented CRSF (with possible loops of length $2$). This is our initial CRSF. We now describe a step by step random popping algorithm of the cycles. Note that at each step, the graph that we see remains an oriented CRSF. Here is the algorithm:
For each cycle $\gamma$ encountered in the current CRSF, pop it with probability proportional to $1-\alpha(\gamma)$; when a cycle
is popped off, the top card on the stacks for each of its vertices is discarded. 
When a cycle is ``kept", its cards are fixed and can no longer be removed.
The algorithm stops once the cycles that remain have been all ``kept" in a Bernoulli trial. It is easy to see that the order in which the cycles are popped is not relevant. 

Note that this algorithm terminates since there is at least one
cycle $\gamma$ with positive weight $\alpha(\gamma)$, because eventually, with probability $1$, 
all the cycles present at one step will have been previously kept (the argument is the same as in Wilson's proof). 

Since the cards of the stacks are distributed as the steps of a conductance-biased random walk, we see that algorithm $\A'$ has the same output in distribution as algorithm $\A$. In order to compute the output distribution of $\A$ we will therefore use algorithm $\A'$.

Let us compute the probability that a given oriented CRSF $\Gamma$ is obtained as an output of algorithm~$\A'$.
Let $\gamma_1,\ldots, \gamma_k$ be the cycles of~$\Gamma$. The CRSF $\Gamma$ is obtained as an output if and only if there exists a finite sequence of oriented cycles $C_1,\cdots,C_m$ such that these cycles are popped, 
and after removing them, the cards that appear correspond to $\Gamma,$ and 
there are $k$ successful trials for Bernoullis with success probability $\alpha(\gamma_i)$. 

By independence of the cards in the stacks, the last CRSF considered 
is independent of the cycles that were popped.
Therefore, for any oriented CRSF $\Gamma$, we have 
\begin{eqnarray*}
\P(\Gamma) &=& \sum_{{\mathcal C}=\{C_1,\dots,C_m\}}\P\left(\Gamma~|~\text{pop }{\mathcal C}\right)\P\left(\text{pop }{\mathcal C}\right)\\
&=&\sum_{\mathcal C}\P\left(\text{pop }{\mathcal C},\text{ and $\Gamma$ occurs underneath and is kept}\right)\\
&=&\sum_{\mathcal C}\P\left(\text{pop }{\mathcal C}\right)\prod_{e\in\Gamma}\P(e)\prod_{i=1}^k\alpha(\gamma_i)\\
&=&\left(\sum_{\mathcal C}\P\left(\text{pop }{\mathcal C}\right)\right)\prod_{e\in\Gamma}\P(e)\prod_{i=1}^k\alpha(\gamma_i)
\end{eqnarray*}
which we see is proportional to the weight of $\Gamma$. 
\end{proof}

To sample a non-oriented CRSF according to a measure which assigns a CRSF~$\Gamma$ a weight proportional to $\prod_{e\in\Gamma}c(e)\prod_{\gamma\subset\Gamma}\alpha(\gamma)$, where the product is over non-oriented cycles~$\gamma$, and $\alpha$ is a function invariant under orientation, it suffices to have $\alpha\in[0,2]$, perform algorithm $\mathcal{A}$ for the measure $\mu_{c, \alpha/2}$, and forget the orientation in the resulting oriented CRSF. In particular, we obtain the uniform measure on non-oriented CRSFs with the choice $c=\alpha=1$.

There is a variant of the previous algorithm to sample an oriented CRSF according to measure $\mu_{c,\alpha}$ conditional on having a single loop: multiply all the loop weights by a small constant $\eps$. 
Then perform $\mathcal{A}$; if $\eps$ is small there will typically be a single loop (if not, start over).  

Let $N$ be the total number of vertices of the graph. The running time of the algorithm is bounded by the time to obtain the first loop 
(which is bounded by $O(N^2)$ if $\alpha\geq O(1/N^2)$) plus the running time of Wilson's algorithm, that is $O(N(\log N)^2)$ (Wilson's algorithm has a running time bounded by the cover time~\cite{Wi} which is linear up to a logarithmic correction~\cite{Al}). The running time is at least linear. Extreme cases correspond to extreme values of $\alpha$: for $\alpha=1$ (uniform measure on CRSFs), the running time is linear; for $\alpha$ close to zero (like in the conditional measure described in the previous paragraph), the running time is large, at least $O(1/\sup \alpha)$. 

\section{Natural probability measures on CRSFs}\label{measures}

The most natural probability measure on CRSFs on a finite unweighted graph is the uniform measure. 
If the edges 
are weighted with a real positive conductance function then in this setting it is natural to give a CRSF a probability
proportional to the product of its edge weights. We call this the \emph{background measure}. 
There are however other natural probability measures that can be constructed from connections on bundles and that are meaningful for graphs embedded in surfaces. 

\subsection{Connections}

Let $\G=(V,E)$ be a finite graph. 
A \emph{vector bundle on $\G$} is a copy $W_v$
of some fixed complex vector space $W$ associated to each vertex $v\in V$. The \emph{total space} of the bundle is the direct sum
$\W=\bigoplus_{v\in V}W_v.$
A \emph{unitary connection} $\Phi$
on $\W$ is the data consisting of, for each oriented edge $e=vv'$, 
a unitary complex linear map 
$\varphi_{vv'}:W_v\to W_{v'}$ such that $\varphi_{v'v}=\varphi^{-1}_{vv'}$.  The map
 $\varphi_{vv'}$ is referred to as the
\emph{parallel transport} from $v$ to $v'$. We say that 
two connections $\Phi,\Phi'$ are \emph{gauge equivalent} if there exist 
unitary $\psi_v:W_v\to W_v$ such that $\psi_{v'}\varphi_{vv'}=\varphi'_{vv'}\psi_{v}$,
that is, $\Phi'$ is obtained from $\Phi$ by changing the basis of each space $W_v$ by a unitary transformation. 
In this paper we deal uniquely with vector bundles with $W=\C$ (line bundles) or $\C^2$,
and $U(1)$- or $\SU$-connections respectively.

Let $c:E\to \R_{>0}$ be a conductance function.
We let $\Delta_\Phi$ be the associated Laplacian
acting on $f\in\W$ defined, for each vertex $v$, by
$$\Delta_\Phi(f)(v) = \sum_{v'\sim v}c(vv')\left(f(v)-\varphi_{v'v}f(v')\right)\,,$$
where the sum is over all neighbors $v'$ of $v$.

When $\G$ is geodesically embedded in a surface $\Sigma$ with a Riemannian metric,
(that is, embedded in such a way that edges are geodesic segments),
there is a natural connection 
$\Phi=\Phi_{\nabla}$ on $\G$ arising from any unitary connection
$\nabla$ on a vector bundle on $\Sigma$: 
we define for each vertex $v$ the space $W_v$ to be the fiber over $v$;
the $\nabla$-parallel transports along edges~$e$ of $\G$ define the parallel transports $\varphi_e$ 
and thus the connection $\Phi$.

The product of parallel transports along a closed path is called the \emph{holonomy} of the connection along the path.
For flat connections it is also called the \emph{monodromy}.

\subsection{Laplacian determinant and measures}
\begin{theorem}[\cite{Fo, Ke1}]\label{detDelta} For a graph with unitary connection $\Phi$ on a line bundle
we have
\be\label{detform}
\det(\Delta_\Phi)=\sum_{\mathrm{CRSFs}}\prod_{\text{edges}}c(e)\prod_{\text{cycles}}(2-2\cos\theta)\,,
\ee
where
$e^{i\theta}$ is the holonomy of the connection around the cycle, for any choice of its orientation. 
\end{theorem}

Associated to $\Phi$ is a probability measure $\mu_{\Phi}$ on CRSFs,
where the probability of a CRSF is proportional to 
$\prod_{\text{edges}}c(e)\prod_{\text{cycles}}(2-2\cos\theta)$. 
This measure exists as long as there is at least one cycle with $\theta\ne 0 \mod 2\pi$.

See Theorem \ref{qdet} below for a generalization to $\C^2$-bundles with $\SU$-connection, where the weight $2-2\cos\theta$ is replaced
by $2-\Tr \,w$, with $w$ denoting the holonomy of the connection around the cycle.
Note that for an element $w\in\SU$ we have $2-\Tr\,w=2-2\cos\theta$ where $e^{\pm i\theta}$ are
the eigenvalues of $w$. 
One can treat a line bundle connection with parallel transports
$\varphi_{e}=e^{i\theta_{e}}$ as a special case of a $\SU$-connection with parallel
transports which are diagonal matrices 
$$\begin{pmatrix}e^{i\theta_{e}}&0\\0&e^{-i\theta_{e}}\end{pmatrix}.$$
The measures on $\SU$-connections are used to analyze
the measures of primary interest $\mu_{nonc}, \mu_{LC}$ and $\mu_{LC^0}$ we discuss below.

\subsection{Graphs with wired boundary}\label{wiredbdry}

Let $B\subset\G$ be a subset of vertices which we consider to be the \emph{wired boundary} of $\G$
or simply \emph{boundary}.
An \emph{essential CRSF} on a graph with wired boundary is a subgraph, each of whose components is either a unicycle
not containing any boundary vertex, or a tree containing a single boundary vertex. 
For a graph with connection and boundary we define
$\Delta_\Phi$ to be the associated Laplacian
acting on $f\in W$ defined, for each vertex $v\in\G\setminus B$, by
$$\Delta_\Phi(f)(v) = \sum_{v'\sim v}c(vv')\left(f(v)-\varphi_{v'v}f(v')\right)\,,$$
where the sum is over all neighbors $v'\in\G$ of $v$ (including neighbors in $B$).
In the natural basis $\Delta_{\Phi}$ is a submatrix (indexed by $\G\setminus B$) of the full laplacian on $\G$.
The analog of Theorem \ref{detDelta} above holds (see \cite{Ke1}) where the sum is over essential CRSFs.

\subsection{Flat connections}

A connection is \emph{flat} if it has trivial holonomy around any contractible cycle.
Suppose that $\G$ is geodesically embedded on a non-simply connected surface $\Sigma$ 
with flat connection $\nabla$.
Let $\Phi=\Phi_{\nabla}$ be the associated connection on $\G$. The associated measure~$\mu_\Phi$ gives zero weight to contractible cycles
and thus is supported on noncontractible CRSFs.

Let $\mu_{nonc}$ be the background measure on noncontractible CRSFs on $\G$ (giving a CRSF a probability proportional
to the product of its edge weights, that is, ignoring any connection). 
The measure $\mu_\Phi$ has density 
$\prod_{\gamma\subset\Gamma}\left(2-\Tr(w_\gamma)\right)$ 
with respect to~$\mu_{nonc}$. 

Although $\mu_{nonc}$ cannot itself be written as a connection measure $\mu_\Phi$
for some flat connection $\Phi$, we can use the $\mu_{\Phi}$ to study $\mu_{nonc}$, see Lemma~\ref{cylinders}
below.

\subsection{Graphs embedded on a curved surface}

\subsubsection{The Levi-Civita measure}

Suppose that $\G$ is geodesically embedded on a Riemannian surface $\Sigma$. There is a natural
complex line bundle on $\Sigma$, the tangent bundle.
Take $\nabla$ to be the Levi-Civita connection on the tangent bundle associated to a metric $g$ on $\Sigma$.
Define $\mu_{LC}$ to be the associated probability measure. It gives a CRSF a probability
proportional to 
$$\prod_e c(e)\prod_{\gamma\subset\Gamma}(2-2\cos\theta_\gamma)$$ where, by the Gauss-Bonnet Theorem,
 $\theta_\gamma$ is the Gaussian curvature
enclosed by $\gamma$. (If $\gamma$ is not contractible it is the ``net turning angle" of $\gamma$.)

\subsubsection{The CRST measure}

When $\Sigma$ is contractible, there is another measure $\mu_{LC^0}$ we can associate to this situation, introduced in \cite{Ke1}.  It
is supported on CRSTs (one-component CRSFs) of $\G$. Let $\Phi=\{e^{i\theta_e}\}_{e\in E}$ be the parallel transports on $\G$
defined from $\nabla$, and for real~$t$
let $\Phi_t=\{e^{it\theta_e}\}_{e\in E}$; these are well defined by contractibility of $\Sigma$. 
Let $\mu_{LC^0}$ be the limit as $t\to 0$ of the measures $\mu_{\Phi_t}$. Since loop weights are going to zero, there will be
only one loop remaining in the limit $t\to 0$, so the limit is a CRST.
In $\mu_{LC^0}$, each CRST $\Gamma$ has a weight proportional to (the product of the edge weights times)~$\theta_\gamma^2$, 
the square of the curvature enclosed by the unique cycle $\gamma$ of $\Gamma$. 

Let $Z_{LC^0}=\sum_{\mathrm{CRSTs}}\theta_\gamma^2\prod_{e}c(e)$ 
be the partition function of $\mu_{LC^0}$. By Theorem \ref{detDelta}, we have 
$$Z_{LC^0}=\lim_{t\to 0}t^{-2}\det\Delta_{\Phi_t}\,.$$ 
An explicit computation of this limit yields the following. Let $\kappa=\kappa(\G)$ be the weighted sum of spanning trees of~$\G$.

\begin{lemma}
We have $Z_{LC^0}=\kappa(\G)\langle\Theta,(I-T)\Theta\rangle$, where $\langle\cdot,\cdot\rangle$ is the usual
scalar product in the space of $1$-forms, $T$ is the transfer current (defined below), and $\Theta$ is the one-form on edges giving the angle $\theta$ of the connection.
\end{lemma}

This lemma also appears as Lemma~2 in~\cite[page 14]{KasWu} but we include the proof here for self-containedness.

The same statement (and proof) applies more generally to the $t\to0$ limit starting from any $U(1)$-connection, not necessarily the Levi-Civita
connection. However we will not have use for these other measures here.

\begin{proof}
The map $T$ is the orthogonal projection onto $\mathrm{Im}(d)$, 
which is the orthocomplement of the space of $\mathrm{Ker}(d^*)$. By a (generalization to varying conductances of a)
result of Biggs (\cite{Big}, Proposition 7.3)
the projection onto $\mathrm{Ker}(d^*)$ is 
given by 
$$\kappa^{-1}\sum_{\text{spanning trees }t}wt(t)\sum_{e\notin t}c_e\pi_{C(e,t)}$$
where $wt(t)=\prod_{e'\in t}c_{e'}$, $C(e,t)$ is the cycle of $t\cup e$ and $\pi_{C(e,t)}$ is the projection onto this cycle. 
Hence
$$T=I-\kappa^{-1}\sum_{\text{spanning trees }t}wt(t)\sum_{e\notin t}c_e\pi_{C(e,t)}.$$
We then have
\begin{eqnarray*}
\kappa\langle\Theta,(I-T)\Theta\rangle&=&\kappa\langle\Theta,{\kappa}^{-1}\sum_t wt(t)\sum_{e\notin t}c_e\pi_{C(e,t)}\Theta\rangle\\
&=&\sum_t wt(t)\sum_{e\notin t}c_e\langle\Theta, \pi_{C(e,t)}\Theta\rangle\\
&=&\sum_t wt(t)\sum_{e\notin t}c_e\frac{\theta_C^2}{|C(e,t)|}\\
&=&\sum_{\text{CRSTs }u}wt(u)\theta_C^2
\end{eqnarray*}
where we used $\langle w,\pi_v w\rangle=\frac{\langle w,v\rangle^2}{\langle v,v\rangle}$ in the second from last equality.
\end{proof}

\subsection{Exact sampling}\label{samples}
The measures $\mu_{nonc}$, $\mu_{LC^0}$, and $\mu_{LC}$ can be sampled using our generalization of
Wilson's algorithm as follows.

The measure $\mu_{nonc}$ is sampled by using a function $\alpha$ which assigns a loop weight~$0$ if it is contractible and $1$ otherwise.

For $\mu_{LC^0}$, we set $\alpha(\gamma)=\eps\theta_\gamma^2$ for small $\eps$. For small enough $\eps$
there will typically be only one loop (if there is more than one, start over).

We can sample $\mu_{LC}$ only in the case where the absolute value of the curvature $\theta$ enclosed by any curve doesn't exceed $\pi/2$. Indeed, in that case, we will always have $2-2\cos\theta\in[0,2]$ which is necessary to sample.

See Figures \ref{LC1}, \ref{LC3}, \ref{inc1} which are obtained by using fine conformal 
approximations to the underlying
surfaces (only the cycles of the CRSFs are drawn). Figures \ref{LC1}, right and \ref{LC3}, right, are samples conditional on enclosing curvature less than $\pi/2$. 
 
In order to sample the unconditional measures, one can use a general algorithm for determinantal processes with Hermitian kernels~\cite{HKPV} which is slower.

Figure \ref{LaminationBranches} shows a sample of $\mu_{nonc}$ on a multiply-connected planar domain. 
\begin{figure}[htbp]
\centering
\includegraphics[width=7.5cm,trim=0 110 0 0]{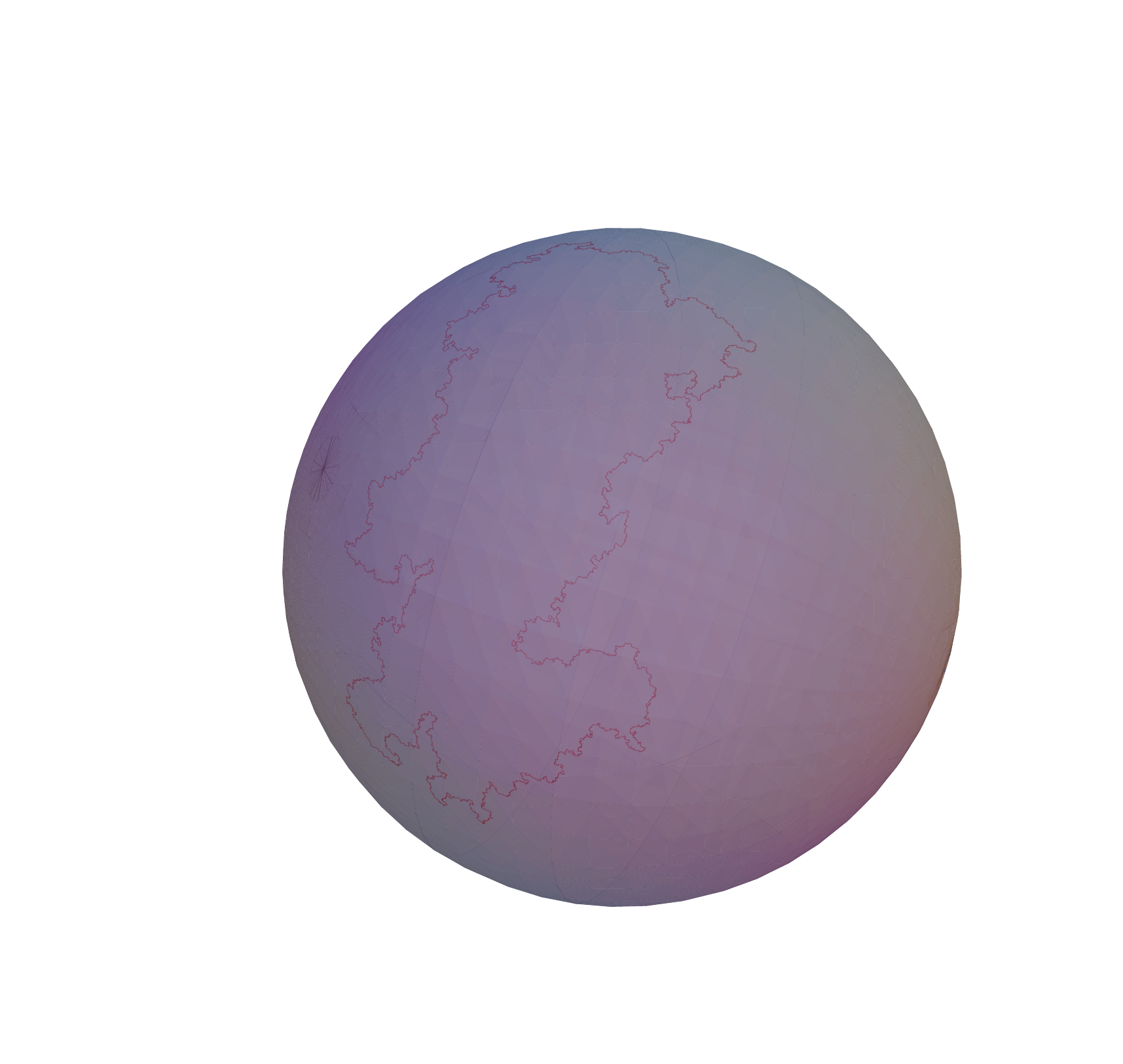}\hskip1cm\includegraphics[width=6.5cm,trim=0 110 0 0]{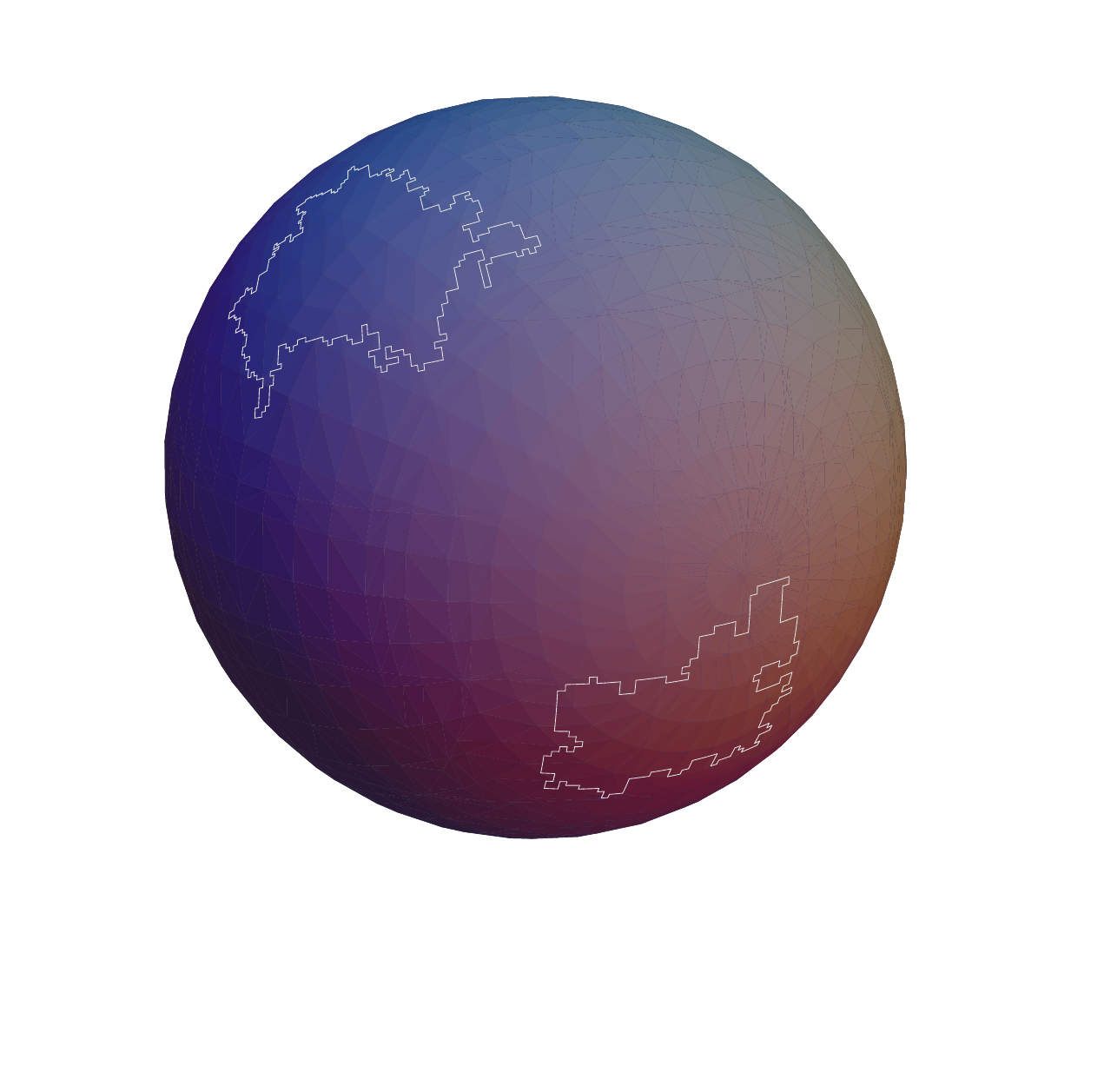}
\caption{
A $\mu_{LC^0}$-random CRSF on the sphere with its round metric; a $\mu_{LC}$-random CRSF on the sphere with two components (a rare event).}\label{LC1}
\end{figure}

\begin{figure}[htbp]
\centering
\includegraphics[width=7cm, trim=0 120 0 0]{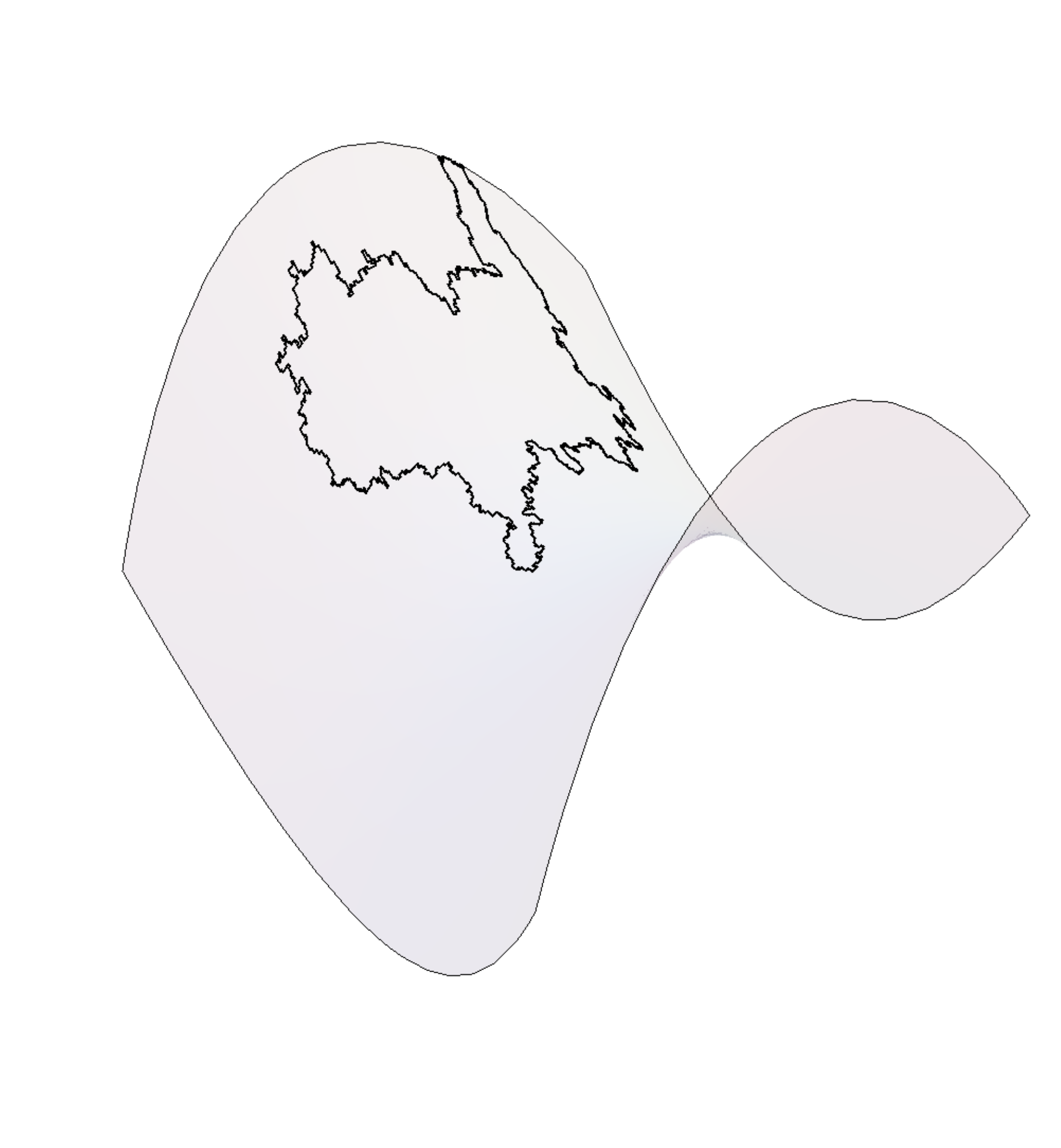}\hskip1cm\includegraphics[width=6cm]{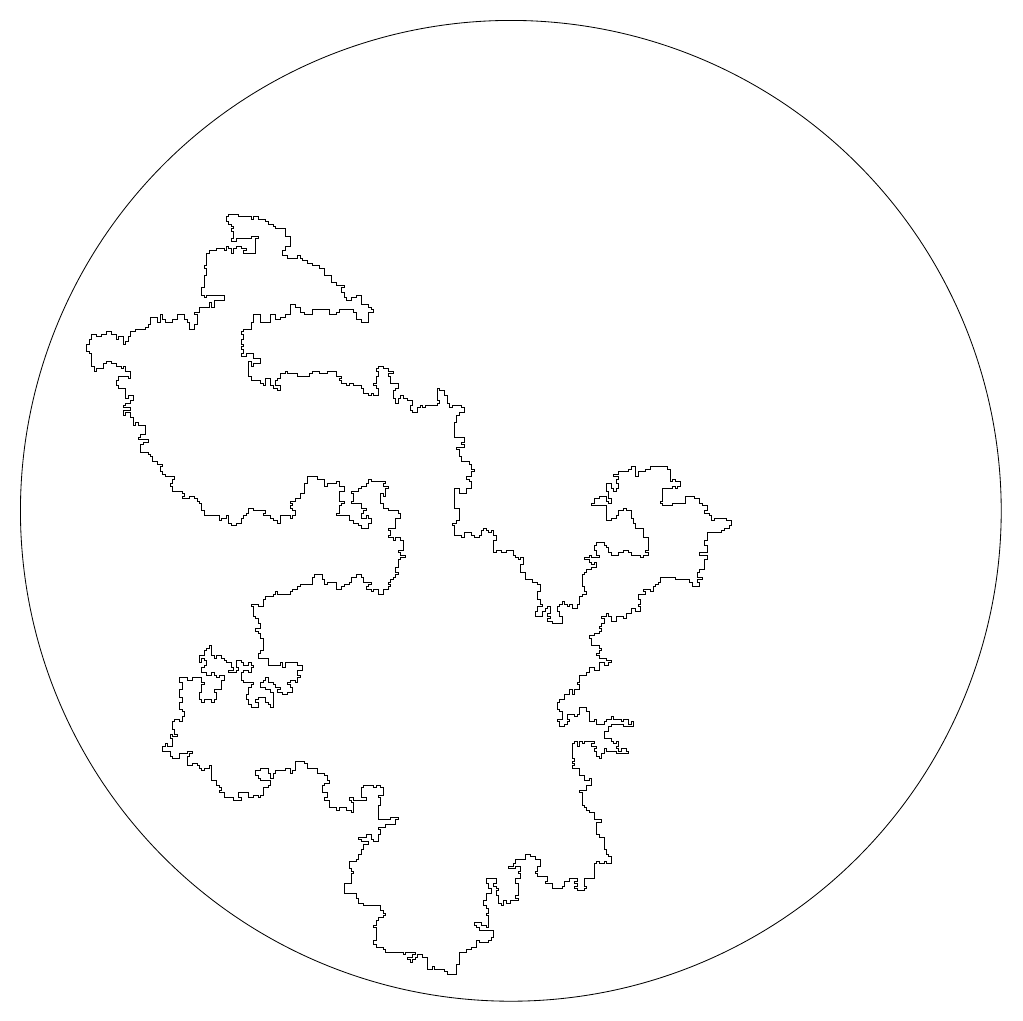}
\caption{A $\mu_{LC}$-random CRSF
on the saddle surface $z=x^2-y^2$ (left), and the hyperbolic plane (right).}\label{LC3}
\end{figure}

\begin{figure}[ht]
\centering
\includegraphics[width=8cm]{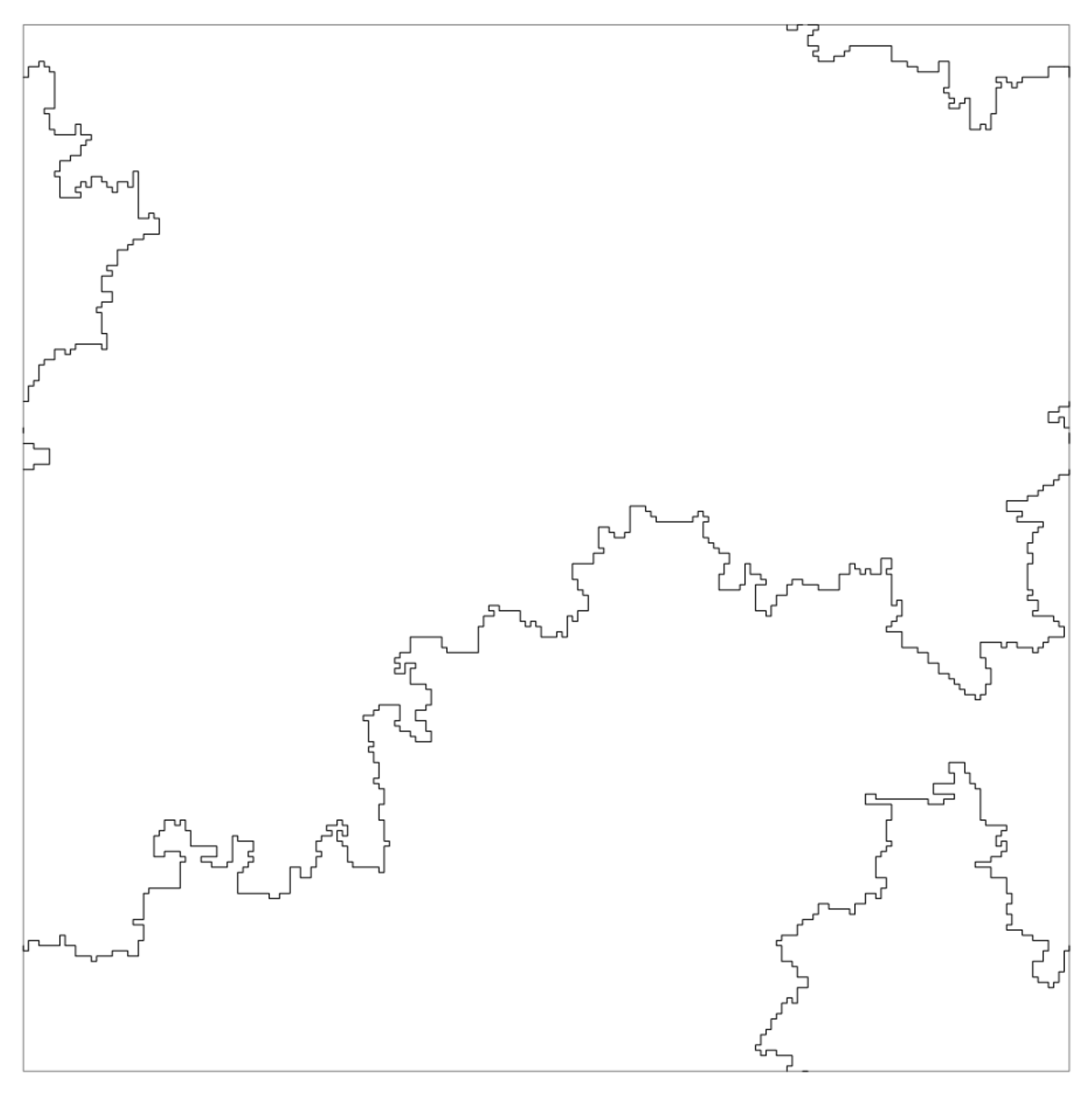}
\caption{A $\mu_{nonc}$-random CRSF on the flat torus
(obtained from the unit square by identifying
opposite sides).}\label{inc1}
\end{figure}

\begin{figure}[htbp]
\centering
\includegraphics[width=8cm]{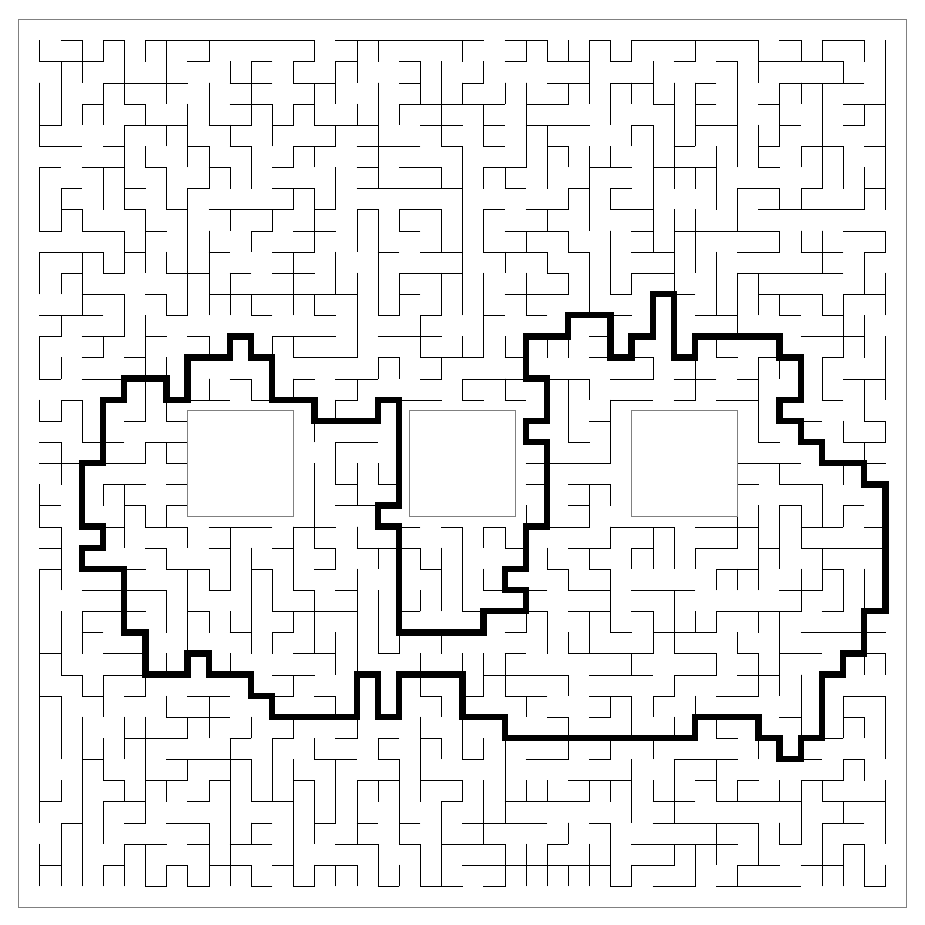}
\caption{A $\mu_{nonc}$-random CRSF on
a punctured disk conditioned on having a particular homotopy type}\label{LaminationBranches}
\end{figure}

\section{Scaling limits for graphs on surfaces}\label{scaling}

The measures $\mu_{nonc},\mu_{LC},\mu_{LC^0}$ induce measures on sets of cycles on $\G$,
by forgetting the rest of the CRSF.
We use notation $\P_{nonc}, \P_{LC}, \P_{LC^0}$ respectively for these cycle measures.

In this section, we prove our main statement (convergence of these cycle measures) in the following way. We first review in~\ref{confapprox} the conformal approximation setup in which we consider the scaling limit. Then in~\ref{metricspace}, we define the probability space in which convergence takes place, namely the metric space of loops up to time-reparameterization. A main ingredient is obtained in~\ref{probanoloopswired} where we show that the probability that there be no loop in the (wired) CRSF measure has a nontrivial limit. This is instrumental in~\ref{tightnesspart} to show that, in the limit, the number of loops remains finite and that the loops are necessarily macroscopic. Combined with earlier technical results on LERW by other authors, this implies tightness of the sequence of measures. We then conclude by Prokhorov's theorem, showing the convergence of a determining class of ``cylindrical'' events defined in~\ref{coeff}. This is done first in~\ref{flatconvergence} in the case of the measure $\P_{nonc}$ using the representation theory of the surface's fundamental group, then extended to flat connections, and finally used to prove the main result, Theorem~\ref{main} in~\ref{curvedconvergence}. The last subsection~\ref{commentandapplications} presents some applications of the main theorem.

\subsection{Conformal Approximation}\label{confapprox}

Let $(\G_n)$ be a sequence of (edge-weighted) graphs geodesically embedded in $\Sigma$ with
mesh size (longest edge length) going to zero.

There are a number of equivalent definitions of the notion of 
conformal approximation of $\Sigma$ by the sequence $(\G_n)$.
Perhaps the easiest is to say that conductance-weighted random walk on $\G_n$ converges to the 
Brownian motion on $\Sigma$, up to reparameterization. Another approach, more computationally useful, 
uses the transfer impedance.

\label{Transferimpedance}
Let $R^n(e,e')$ be the transfer impedance
of two oriented edges $e=xy$ and $e'=x'y'$ in $\G_n$.
This is defined as the potential drop across $e'$ when one unit of current enters at $e_-$ and leaves at $e_+$ (the endpoints of $e$), 
when the graph is viewed as an electrical network with conductance $c$.
In terms of the Green function $G^n$ one has 
$$R^n(e,e')=G^n(e_+,e'_+)-G^n(e_+,e'_-)-G^n(e_-,e'_+)+G^n(e_-,e'_-).$$
A related quantity is $T(e,e')=c(e')R(e,e')$ the \emph{transfer current}: the current across
$e'$ when one unit of current enters at $e_-$ and leaves at $e_+$. 

The function $R^n(e,e'),$ as a function of either $e$ or $e'$, 
is a discrete one-form on $\G_n$ (function on oriented edges which changes sign under change
of orientation). We say that \emph{conformal approximation holds} if the mesh size tends to zero as $n\to\infty$ and,
when $e_1,e_2$ are not within $o(1)$ of each other, 
\begin{equation}\label{conformalapproximation}
R^{n}(e_1,e_2)=(D_{e_1}D_{e_2}g(z_1,z_2))\ell(e_1)\ell(e_2)+o(\ell(e_1)\ell(e_2)\,,
\end{equation}
where $g$ is the continuous Green function, $\ell(e_i)$ is the edge length,
and the notation~$D_{e_i}g$ represents the directional  
derivative in the direction of the edge $e_i$ applied to the variable~$z_i$, that is
$$D_{e_1}g(z_1,z_2) = \lim_{\delta\to0} \frac{g(z_1+\delta u_1,z_2)-g(z_1,z_2)}{\delta}$$
where $u_1$ is the unit displacement in the direction of $e_1$,
and similarly for $D_{e_2}$ for the second variable $z_2$.

If we sum the transfer impedance for $e_1$ varying along a path
$\gamma_{a,b}$ from vertex $a$ to $b$, and $e_2$ on $\gamma_{c,d}$ from $c$ to $d$, we find that 
\begin{multline}\sum_{e_1\in \gamma_{a,b}}\sum_{e_2\in\gamma_{c,d}}G^n(e_1^-,e_2^-)-G^n(e_1^-,e_2^+)-G^n(e_1^+,e_2^-)+G^n(e_1^+,e_2^+)=\\g(a,c)-g(b,c)-g(a,d)+g(b,d)+o(1).\end{multline}
This quantity represents the change in potential from $c$ to $d$ when one unit of current
enters at $a$ and exits at $b$. As a function $F(c)$ of $c$, this is the unique harmonic function with 
appropriate logarithmic singularities
at $a$ and $b$, and which is zero at~$d$. In particular the level curves of $F$ are equipotentials and,
taking equipotentials near $a$ and $b$ (which are close to circles), one can thus 
compute the resistance between a small circle around $a$ and a small circle around $b$.
In this way, using convergence of the transfer impedance, one can show that the main notions of potential theory including
the Poisson kernel, Cauchy kernel, holomorphic functions, {\it etc.}\! all behave well under conformal approximation.

As an example of a conformally approximating sequence, standard arguments show that the square
grid in $\C$, scaled by $\eps$, conformally approximates any planar domain in $\C$ as $\eps\to0$. 
Thus for a simply connected domain in $\Sigma$
we can pull back a fine square grid in $\C$ under a conformal
map from a domain in $\C$ to $\Sigma$. 
More generally the (almost-)equally-spaced real and imaginary curves of a holomorphic quadratic differential 
$\phi(z)dz^2$ on $\Sigma$ give
a graph structure on $\Sigma$, with unit conductances, conformally approximating the surface. 
Other examples from Poisson point processes are given in \cite{GK}.

\subsection{The probability space of simple multiloops}\label{metricspace}

For a graph $\G_n$, let $\P^n$ be one of the measures on loops discussed above. 
We can see the probability measures $\P^n$ for different $n$s as living on the same probability space~$\Omega$ that we now describe. 

A \emph{multiloop} in $\Sigma$ is a finite union of simple loops (ignoring parameterization) with disjoint images, that is,
an injective map from the union of $k$ copies of the unit circle $\T$ to $\Sigma$, 
for some $k\ge1$ (and modulo reparameterizations). 
The space of single simple loops is a metric space: the distance is defined by
$$d(f,g) = \inf_{\alpha}\sup_{t\in S^1}\|f(t)-g(\alpha(t))\|$$
where the infimum is over all reparameterizations $\alpha$. In other words two loops are close if they can be parameterized
so that their images are close for all parameter values $t$. One can extend this distance to the case of multiloops by
taking the infimum over all permutations of loops, of the max of the 
pairwise distances (and defining the distance between $\Omega_k$ and 
$\Omega_{k'}$ for $k\ne k'$ to be infinite).
With this distance $\Omega$ is a topological metric space. It is a disjoint union $\Omega=\cup_{k=1}^\infty\Omega_k$ where $\Omega_k$ consists of multiloops with $k$ components.

This space is not complete: it is easy to construct Cauchy sequences that shrink to a point or to non-simple loops. 
However it is separable: take all finite multiloops lying on fine lattice approximations of the surface. This is a countable family of loops which is dense in~$\Omega$.

The set of cycles of a CRSF on $\G_n$ defines an element of $\Omega$. 

A \emph{finite lamination} on a surface is an isotopy class of a multiloop.
For any points $x_1,\ldots, x_m\in\Sigma$, and small $\delta>0$, let $B_i$ be a ball around $x_i$ of some radius less than $\delta$, and consider the finite laminations in $\Sigma\setminus\{B_1\cup\cdots\cup B_m\}$. Any 
multiloop~$\gamma$ which avoids the balls $B_i$ defines a lamination $[\gamma]$. 
For any of these laminations $L$ we consider the event

$$E_{B_1,\dots,B_m;L}=\left\{\gamma\in\Omega\quad{}~|~\quad{} [\gamma]=L\right\}\,.$$

We call these sets \emph{cylindrical events} and consider the $\sigma$-field $\B$ on $\Omega$ generated by these events. 

\begin{lemma}$\B$ contains the Borel sets in $\Omega$. 
\end{lemma}

\begin{proof} 
We just prove this for one loop, that is, for $\Omega_1$. The proof is easily extended to the general case.

Let $c$ be a smooth simple closed curve in $\Sigma$ and for some small $\eps>0$ let
$U_\eps(c)$ be its $\eps$-neighborhood.
Consider the event $E$ that the random curve 
$\gamma$ maps into $U_\eps(c)$, winding once around the annulus with, say, the positive orientation. 
These types of events generate the Borel sets.

Let $x_1,x_2,\dots$ be a sequence of points dense in the boundary of $U_\eps(c)$. Let $B_i$ be a ball
around $x_i$ of radius $\delta/i$. 
Let $E_{\delta;n}$ be the cylinder event that the curve $\gamma$
in $\Sigma\setminus\{B_1\cup\cdots\cup B_m\}$
separates the points $\{x_1,\dots,x_n\}$ on the
two boundary components of $U_\eps(c)$, that is, is consistent with $\gamma$ winding once around $U_\eps(c)$. 
The event $E$ is contained in the intersection over $n$ of the $E_{\delta;n}$.
In fact we have $E=\cap_{n}E_{0;n}=\cup_{\delta\to 0}\cap_n E_{\delta;n}$: any continuous simple loop which separates the points lies in the
interior of the annulus and winds once around. This can be seen as follows. First of all, any curve in $\cap_{n}E_{0;n}$ is contained in $U_\eps(c)$: otherwise, we could find some $x_j$ lying on the wrong side of the curve since the family $(x_n)_{n\geq 1}$ is dense in the boundary which would yield a contradiction. Second, the curve cannot be contractible since this would contradict the fact that it separates the points. Hence it winds once around the hole of the annulus. Its orientation is necessarily positive. 
\end{proof}

\begin{lemma}
The set of cylindrical events forms a determining class for $\B$.
\end{lemma}
\begin{proof}
This class is stable under finite intersection and generates the sigma-field.
\end{proof}

\subsection{The probability of no loops}\label{probanoloopswired}

Recall that on a graph with boundary, an essential CRSF is a subgraph, each of whose components is either a unicycle
not containing any boundary vertex, or a tree containing a single boundary vertex. 
For a graph with boundary and flat connection~$\Phi$, the probability that a $\mu_{\Phi}$-random
essential CRSF has no loops tends to $1$ 
as the holonomy tends to the identity, since each noncontractible cycle has weight $2-2\cos\theta$ 
which tends to zero.

For $\mu_{LC}$, the following theorem describes a similar result. For a graph embedded on a surface $\Sigma$ with boundary,
we define the boundary of the graph to be the set of vertices adjacent to the boundary of $\Sigma$. Let $\{\G^n\}_{n=1,2,\dots}$ be a sequence of graphs conformally approximating $\Sigma$, where $\G^n$ has mesh size at most $1/n$.

The intuition behind the following theorem is that one may express the quantities as functionals of the random walk loop soup and observe that the random walk loop soup converges to the Brownian one, see e.g. \cite{LF}.

\begin{theorem}\label{universality}
For any compact surface with nonempty boundary and any smooth unitary connection $\nabla$ on a line bundle on $\Sigma$, the probability $\P^n(\text{no loops})$ converges to a universal limit $p(\Sigma)$.
\end{theorem}

\begin{proof}
Let $\Delta_{\text{Id}}$ be the Laplacian on $\G^n$ for the trivial connection.
We have
$$
\P^n(\text{no loops})=\frac{\det \Delta_{\text{Id}}}{\det \Delta_\Phi}.$$
Write $\Delta_{\Phi} = D(I-P_\Phi)$ where $D$ is the diagonal ``degree" (sum of weights) matrix.
Then 
\begin{eqnarray}\nonumber
-\log \P^n(\text{no loops})&=&-{\Tr}\log (I-P_{\text{Id}})+\Tr\log (I-P_{\Phi}). 
\\
\label{int} &=& \sum_{k=1}^\infty \frac{\Tr P_{\text{Id}}^k}{k}-\frac{\Tr P_\Phi^k}{k}\\
\nonumber &=& \sum_{\text{loops~}\ell}\P(\ell)(1-\cos\theta_\ell)
\end{eqnarray}
where the sum is over unrooted loops $\ell$, and where $\P(\ell)$ is the probability of $\ell$ (for the weighted random walk started at some vertex of $\ell$) and $e^{i\theta_\ell}$ is the holonomy of the loop $\ell$.

We divide this sum into three parts: tiny loops (consisting of at most $\eps\sqrt{n}$ steps), small loops
(consisting of at most $\eps n^2$ steps),
and large loops (at least $\eps n^2$ steps).

For tiny loops of $k$ steps, the area enclosed is at most $O(k^2/n^2)$ by the isoperimetric inequality.
Since $\theta_\ell$ is at most a constant times the enclosed area,
the contribution for tiny loops is bounded by 
$$\sum_{k=1}^{\eps\sqrt{n}} \frac{\Tr(P^k)}{k}O\left((k^2/n^2)^2\right).$$
Since $P$ is substochastic, $\text{Tr}(P^k)=O(n^2),$ and the total contribution is $O(\eps^2)$.

For small loops of length $k$, we argue that they enclose signed area $O(k/n^2)$.
We use a small generalization of the following result of Wehn, \cite{Wehn}: 
a two-dimensional simple random walk of length $k$,
conditioned to return to the origin, encloses a signed area (that is, the integral of $x\,dy$) of order $O(N)$, that is, has standard deviation $O(N)$. The argument is as follows.
The signed area is obtained from an integral $\int x\,dy$ along the loop 
(where $(x,y)$ are local orthogonal coordinates with origin at the starting point).
Each step makes an essentially independent contribution to the integral (the only dependence being the 
condition that the loop is closed after $k$ steps). By grouping several steps at a time, the mean contribution
for each group is zero, whereas the variance is of order $k/n^4$, since $x=O(\sqrt{k}/n)$ and $dy=O(1/n)$.   
Summing the variance over the $k$ steps, one gets a total variance $O(k^2/n^4)$.
Using the fact that a random walk of length $k$ returns to its starting point with
probability $O(1/k)$, the contribution for small loops is
$$\sum_{k=1}^{\eps n^2} \frac{\Tr(P^k)}{k}O(k^2/n^4)=\sum_{k=1}^{\eps n^2} O\left(\frac{n^2}{k^2}\right)O(k^2/n^4)=O(\eps).$$

By~\cite{LF}, the sum over large loops converges to the analogous continuous loop measure. This is because long loops can be approximated with Brownian excursions.
Along such an excursion the parallel transport is approximated by the Brownian
parallel transport\footnote{The \emph{Brownian parallel transport} is defined as a limit of the parallel transport along piecewise geodesic approximations to
the Brownian motion: for a Brownian path $B(t)_{t\in[0,1]}$, take for example for $N$ large the piecewise geodesic path connecting the points $B(j/N)$ and $B((j+1)/N)$ for $j\in\{0,N-1\}$. It\^o showed in \cite{Ito} that almost surely the parallel transports along the piecewise geodesic path converge in the limit $N\to\infty$ to a quantity
independent of the approximating discrete path.}.   This Brownian parallel
transport is a universal quantity in the sense of being independent of the underlying graph, only depending on the underlying smooth connection.
\end{proof}

Suppose now $\nabla$ is close to the identity: that is, for some small $c>0$, 
the integral of the absolute value of the curvature of~$\nabla$ is less than $c$, 
and the holonomies $w$ of~$\nabla$ on a fixed cycle basis satisfy $|w-1|<c$.

\begin{corollary}\label{noloopatall}For $\nabla$ close to the identity in the above sense, the probability of no loops satisfies
$$\P^n(\text{no loops})=1+o(1)$$
where the error is independent of $n$.
\end{corollary}

\begin{proof}
In the proof of the above theorem, with a total curvature bound $c$, the term $1-\cos\theta_\ell$ is of
order $c^2O(\theta_\ell^2).$
Thus the ratio of determinants, and hence the probability of no loops, is $1-O(c^2)$. \end{proof}

\subsection{Asymptotic size of loops and tightness of the measures}\label{tightnesspart}

\subsubsection{Macroscopic loops}
For $\P_{nonc}^n$, the loops necessarily are macroscopic since they are noncontractible. 
In the curved case for the measures $\P_{LC}^n$ and $\P_{LC^0}^n$, we show that there are, with positive probability,
macroscopic loops.

\begin{theorem}\label{nonzero} With positive probability $\P_{LC}^n$ and $\P_{LC^0}^n$ contain a macroscopic loop, that is,
for sufficiently small $\eps>0$ the probability that there is a loop with diameter $\ge\eps$ does not tend to zero with $n$. 
\end{theorem}

\begin{proof}
We can suppose the surface is a disk: if not, take a point of nonzero curvature on the surface and a disk around it, small enough so that the curvature is roughly constant on the disk.
Wiring the boundary of this disk, and making the disk boundary part of the surface boundary (effectively cutting the surface apart along the disk boundary) 
\emph{decreases} the probability of finding a macroscopic loop in this disk, since the loop-erased walk from any interior point will halt sooner. So once we prove the theorem for a disk 
with nonzero curvature we are done.

We will prove the theorem for $\P_{LC^0}^n$. Since we
are assuming at this point that the surface is a disk, this will also suffice for $\P_{LC}^n$
by the following argument. Choose the disk small enough so that the probability of no loops is at least $1/2$.
Then for each loop discovered by the algorithm, the conditional probability of having no further loops is at least $1/2$
(since each new loop discovered decreases the probability of further loops).
Thus the number of loops is smaller than a geometric random variable of rate $1/2$, and the probability of $1$
loop strictly dominates the probability of more than one loop. The change of loop weight between $\P_{LC}^n$
and $\P_{LC^0}^n$ tends to~$1$ with small curvature $\theta$, so these are absolutely continuous with a Radon-Nikodym
derivative independent of $n$. 

Now consider the case of $\P_{LC^0}^n$.
For this, consider first the $n\times n$ grid $H_n$ scaled by $1/n$ to the unit square $[0,1]^2$.
Let $z_1\ne z_2\in(0,1)^2$ and let $f_1,f_2$ be faces of the grid close to $z_1,z_2$ respectively.
In \cite{KKW}, it is shown that
the probability $\P(f_1,f_2)$ that $f_1,f_2$ are enclosed in the cycle of a uniform CRST of $H_n$
is $const/n^2+o(1/n^2)$.
Let $E$ be the event that $f_1,f_2$ are enclosed in the cycle. On this event, the area of the enclosing cycle
is with high probability $\ge \eps n^2$ for some $\eps>0$ (see below, where it is shown that the loops are absolutely continuous with
$\SLE_2$). 
Then 
\be\label{PLC0}\P_{LC^0}^n(E) = \frac{\sum_E \text{Area}^2}{\sum_{CRSTs} \text{Area}^2}\ge 
\frac{\eps^2 n^4\sum_E1}{Cn^2\sum_{CRSTs}1}=\eps^2 C'n^2\P(f_1,f_2)
\ee
for constants $C,C'$. Here the denominator of the central inequality follows from the result of \cite{KKW}
that the second moment of the area of a uniform CRST is of order~$n^2$ times a constant. 
The right-hand side of (\ref{PLC0}) is bounded below by a positive constant independently of $n$.

A similar argument holds for any sequence of graphs $(\G_n)$ conformally approximating a curved disk $\Sigma$, and near a point where
the curvature of the metric is nonzero: In a small neighborhood $U$ of such a point, the curvature is approximately constant and
the transfer impedance for $\eps\Phi$ will thus agree to first order with that in a similar small neighborhood of a point in $H_n$
and far from the boundary of $H_n$. 
Since $\P(f_1,f_2)$ is determined by the 
Green's function on the dual graph \cite{KKW}, 
for $f_1,f_2\in U$ we have a similar estimation of $\P_{LC^0}^n(E)$ as in the case for $H_n$ above.
\end{proof}

\subsubsection{Number of loops}\label{nbrofloops}

We show here that the number of loops has super-exponential decay for both $\P_{nonc}^n$ and $\P_{LC}^n$.

\begin{lemma}\label{decay} For
any sequence $(\G_n)_{n\ge1}$ of graphs conformally approximating a compact Riemannian surface $\Sigma$,
possibly with boundary, we have the following.
For any $0<\xi<1$, there exists $N$ and $K$ such that for all $n\geq N$ and for all $k\geq K$, we have
$$\P^n\left(\text{there are at least $k$ loops}\right)\leq \xi^k\,.$$
\end{lemma}

\begin{proof} 
Each loop created during the performance of the cycle-popping algorithm
either disconnects the surface or decreases the rank of the first homology. 
By the Markov property (Section~\ref{Markov}), the law of the conditional CRSF is obtained by independently sampling in each of the connected components.  

For $\P^n_{LC}$, suppose that we have created $k$ closed curves. Each of the complementary
components is either planar or non-planar. Each new loop found has a positive probability of being macroscopic, that is, either reduces
the rank of the first homology or removes definite area
from both resulting components, by Theorem \ref{nonzero}. Thus for $k$ sufficiently large, the curvature enclosed $c$ on either side 
will be small, hence with definite probability $1-p=1-o(1)$ the two sides will contain no further loops, by Corollary~\ref{noloopatall}. By taking $k$ large enough, and mesh size $1/n$ small enough, the curvature enclosed $c$ can be taken as small as needed such that $p<\xi$.

Thus the probability of an extra loop eventually decays exponentially with rate less than~$\xi$.

A similar argument works for $\P_{nonc}$.
\end{proof}

As an example, the distribution of loops for $\P_{nonc}$ with free boundary conditions was computed for an annulus using the standard square grid approximation in~\cite{Ke1}. For a cylinder of aspect ratio (height to circumference) $\tau$, the probability generating function of the number of loops is an elliptic function
$$\E^{free}_{\tau}(X^{\# loops})=X \prod_{j=1}^{\infty}\frac{X+2\cosh{(\pi j /\tau)}-2}{2\cosh{(\pi j /\tau)}-1}\,,$$
which can be checked to have a super-exponential tail.
For wired boundary conditions, and by planar duality (since the dual of an essential incompressible CRSF is a CRSF with free boundary conditions) the distribution is $\E^{wired}_{\tau}(X^{\# loops})=X^{-1}\E^{free}_{\tau}(X^{\# loops})$, which naturally also has a super-exponential tail.

\subsubsection{Microscopic loops}

The following theorem shows that loops of $\P^n_{LC}$ do not shrink to points as $n\to\infty$.

\begin{theorem}\label{nomicroloops}
Any subsequential limit of $\P^n_{LC}$ as $n\to\infty$ is supported on $\Omega$, 
that is, as $\alpha\to0$
\be\label{nopoint}\limsup_{n\to\infty} \P^n_{LC}\left(\text{there is a loop of area $\le\alpha$}\right)=o(1)\,.\ee
\end{theorem}

\begin{proof}
Let us argue by contradiction.  
If we suppose that~\eqref{nopoint} is false, it means that there exists $p>0$ and arbitrarily small values of $\delta>0$ such that for $n$ large enough, we have
$$\P^n_{LC}\left(\text{there is a loop with area $\leq \delta^2$}\right)>p\,.$$

A small-area loop has a small diameter; conditioning on this loop in particular does not change
the transfer impedance operator far from that loop. In fact from Lemma \ref{locallyLERW}
 below, the loops
are absolutely continuous with respect to the loop-erased walk; as such their
probability of intersecting a small disk tends to zero with the disk's diameter, since the same is true for the loop-erased walk. Thus the 
conditioning has negligible effect.

So we can expect to find many loops: there is a $0<q<p$ such that,
dividing $\Sigma$ into regions of diameter $\sqrt{\delta}$,
the probability that there is a loop in each of these $\Theta(1/\delta)$ regions is on the order of $q^{1/\delta}$. Thus
$$\P^n_{LC}\left(\text{there are $1/\delta$ loops with area $\leq \delta^2$}\right)>q^{1/\delta}\,.$$
However, by Lemma \ref{decay} above, 
there exists $0<\xi<q$ such that for any $k$ and $n$ large enough, we have
$$\P^n_{LC}\left(\text{there are $\geq k$ loops}\right)\leq \xi^k\,.$$
By taking $k=1/\delta$ we obtain that for arbitrarily small values of $\delta>0$, there is a large enough $n$ such that 
$$0<q^{1/\delta}<\P^n_{LC}\left(\text{there are $1/\delta$ loops with area $\leq \delta^2$}\right)\leq \xi^{1/\delta}\,.$$
This yields a contradiction since the right-hand side of the last equation tends to zero faster than the left-hand side when $\delta\to 0$.
\end{proof}

\subsubsection{Resampling and tightness}
We will show tightness of the sequence of measures $\P^n_{nonc},\P^n_{LC^0}$ and $\P^n_{LC}$. This will yield the existence of subsequential limits. 

We show in fact that the scaling limits of the macroscopic loops are absolutely continuous with respect to $\SLE_2$. For this we establish the link to the loop-erased random walk (LERW).

\begin{lemma}[Link to LERW]\label{locallyLERW}
Let $\eta$ be a simple path from vertices $\eta^+$ to $\eta^-$. The law of the cycle $\gamma$ 
of a uniform CRST, conditional on the fact that it contains $\eta$, is ($\eta$ followed by) the 
LERW from $\eta^+$ to $\eta^-$ with wired boundary conditions along $\eta$.

For the measures $\P_{LC}, \P_{LC^0}$ or $\P_{\text{nonc}}$, conditional on all other components, the law of $\gamma$
is asymptotically that of ($\eta$ followed by) the LERW from $\eta^+$ to $\eta^-$ with wired boundary conditions along $\eta$,
biased by the cycle weight of $\gamma$.
\end{lemma}

\begin{proof}
Let $\gamma$ be a loop in a $\P^n_{LC^0}$-,$\P^n_{LC}$- or $\P^n_{nonc}$-random CRST, 
and $a,b\in\gamma$ distinct points on it. Let $\gamma[a,b]$ be the part of $\gamma$ 
counterclockwise
between $a$ and $b$, and $\gamma[b,a]$ the complementary part. If we erase $\gamma[a,b]$,
we can define a new loop $\gamma'$ by taking a LERW from $a$ to $b$ in the domain
defined by $\Sigma\cup\gamma[b,a]$, with wired boundary conditions on $\gamma[b,a]$, and conditioning on the LERW
to end at $b$. The union of this LERW and $\gamma[b,a]$ is the simple closed curve $\gamma'$. 
By the sampling algorithm, this curve $\gamma'$ is absolutely continuous with  $\gamma[a,b]$,
with Radon-Nikodym derivative given by the ratio of the cycle weights.
If $a$ and $b$ are close to each other, the
LERW from $a$ to $b$ will with high probability not exit a small ball around $[a,b]$. Hence the cycle weight of
$\gamma'$ will be close to that of $\gamma$. 

Wilson's algorithm thus shows that
this is a fair sample of $\P$ conditioned on $\gamma[b,a]$. Thus the LERW from $a$ to $b$ with the appropriate
boundary conditions is absolutely continuous with respect to $\gamma[a,b]$, with a bound independent of mesh size~$1/n$.
\end{proof}

This proves that in the scaling limit, for any converging subsequence, the loops are locally absolutely continuous with respect to the scaling limit of LERW, which was shown in \cite{LSW} to be $\SLE_2$.

In particular the scaling limit, for any converging subsequence, is supported on simple curves (recall that $\SLE_2$ curves are simple, and note that
this is a local property). This tightness property of LERW was proved in~\cite{AB,ABNW,Sc,LSW}.

\begin{theorem}\label{tightness}
The sequences $\left(\P^n_{nonc}\right)_{n\geq 0}$ and $\left(\P^n_{LC}\right)_{n\geq 0}$ are tight on $\Omega$.
\end{theorem}

The above argument also shows that the sequence $(\P^n_{LC^0})_{n\ge 0}$ is tight, 
provided we allow for the possibility that the curve shrinks to a point.
We need to add to $\Omega_1$ a copy of $\Sigma$ whose points represent the constant maps of $S^1$ to that point.
The metric naturally extends to this augmented space $\Omega_1^*$, so that 
the limit of a sequence of curves shrinking to a point is the constant
curve at that point. Let $\Omega_1^*$ be this augmented space.

\begin{theorem}\label{tightness0}
The sequence $\left(\P^n_{LC^0}\right)_{n\geq 0}$ is tight on $\Omega_1^*$.
\end{theorem}

\subsection{Probabilities of cylindrical events}\label{coeff}

\subsubsection{Flat connections}

Let $\Sigma$ be a compact surface with $b>0$ boundary components.
Let $x_1,\ldots,x_k\in\Sigma$. 
Let $\M$ be the space of flat $\SU$-connections modulo gauge transformations on $\Sigma\setminus\{x_1,\ldots,x_k\}$. 
This space is compact. 
Such a flat connection is determined uniquely by a homomorphism from $\pi_1\left(\Sigma\setminus\{x_1,\ldots,x_k\}\right)$ into $\SU$ modulo conjugation. 

Provided $b+k>0$, the fundamental group $\pi_1(\Sigma\setminus\{x_1,\dots,x_k\})$ is a free group $F_m$ 
on $m=2g+k-1+b$ letters, where $g$ is the genus
and $b$ the number of boundary components of $\Sigma$. Thus a flat connection is 
determined by $m$ arbitrary elements of $\SU$, one for each generator of $\pi_1$. 
 
There is no canonical basis for $\pi_1(\Sigma\setminus\{x_1,\dots,x_k\})$ and hence for $\M$. For any choice of a basis, we consider the measure on $\M$ obtained by the image of the Haar measure on $\SU^m$. It can be seen, using a theorem of Nielsen (see~\cite{LS}, Chapter~$1$), that the measure is independent of that choice of basis. (This theorem states that one can go from one basis to another in a free group by a sequence of elementary moves. It is easy to check that these moves preserve the Haar measure.) We let $\nu$ be this measure and call it the canonical Haar measure on $\M$.

\subsubsection{Trivalent graph}
A useful device, see \cite{FG}, is to define a trivalent graph $H$ in $\Sigma$ (unrelated to $\G$) 
with a single boundary component of $\Sigma$ or $x_i$ in each face,
so that $H$ is a deformation retract of $\Sigma\setminus\{x_1,\dots,x_k\}$. We can think of $\Sigma$ as a ribbon graph structure on $H$.

Recall that a \emph{finite lamination} $L$ of $\Sigma$ is an isotopy class of a finite number of disjoint noncontractible
simple closed curves.
A lamination $L$ retracts to a ``multicurve with multiplicity" on $H$; $L$ is determined by, for each edge of $H$, a nonnegative
integer giving the number of strands of a minimal representative of $L$ retracted onto that edge. These integers satisfy
the conditions that at each vertex of $H$ the sum of the three integers is even and the three integers
satisfy the triangle inequality: see Figure~\ref{trivalent}. Moreover any set of integers satisfying these two conditions
arises from a unique lamination.

\begin{figure}[ht]
\centering
\includegraphics[width=10cm]{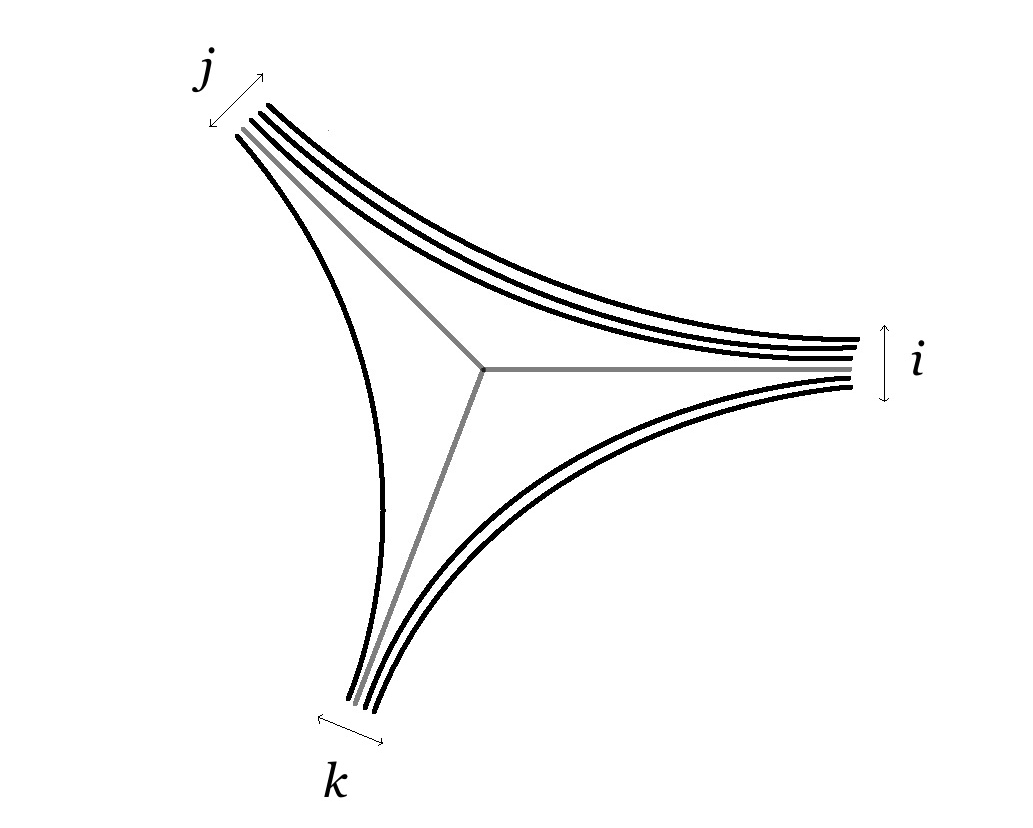}
\caption{Strands of the lamination at a vertex of the trivalent graph; the integers $i,j,k$ (here $5,4,3$)
satisfy $i+j+k=0\mod 2$ and the triangle inequality.}
\label{trivalent}
\end{figure}

We define a partial order on laminations: $L\leq L'$ if the integers on edges of $H$ associated to $L$ are
all less than or equal to those of $L'$. We define the \emph{complexity} $n(L)$ of $L$ to be the sum of these integers.

An $\SU$-connection on $H$ determines a flat connection on $\Sigma$. After gauge transformation
one can take the $\SU$-connection to be the identity on all edges of a spanning tree of $H$. 

\subsubsection{Integrals over the space of flat connections}\label{integrals}
Given a finite weighted graph $\G$ embedded in $\Sigma$, and $\Phi\in\M$ define
\be\label{Zphi}
Z(\Phi)= \sum_{\text{CRSFs}}\prod_{\text{edges}}c(e)\prod_{\text{cycles}~\gamma}\left(2-\Tr(\omega_\gamma)\right),\ee
where
$\omega_\gamma$ is the holonomy of the connection $\Phi$ around the cycle $\gamma$. The function 
$Z$ is a real-valued function on $\M$. 

We denote $$Z_0=\sum_{\text{nonc. CRSFs}}\,\prod_{\text{edges}}c(e),$$ the partition function for all noncontractible CRSFs, without cycle weight. (This is \emph{not} the same as $Z(\text{Id})$, which is zero.)

For a flat connection $\Phi$ we may rewrite (\ref{Zphi}) as 
$$Z(\Phi)=\sum_{L}X_{L}T_{L}(\Phi),$$
where the sum is over finite laminations $L$,
where $X_{L}$ is the conductance-weighted sum of CRSFs whose cycles are isotopic to $L$, and 
$$T_{L}(\Phi)=\prod_{\gamma\in L}\left(2-\Tr(\omega_\gamma)\right)$$ 
is a real-valued function on $\M$.

Fock and Goncharov \cite{FG} proved that,
seen as real-valued functions on $\M$, the functions $S_{L}=\prod_{\gamma\in L}\Tr(\omega_\gamma)$ are linearly independent and generate the vector space of regular (polynomial)
functions on $\M$, when $L$ runs through all finite laminations.
Hence the functions $T_{L}=\prod_{\gamma\in L}\left(2-\Tr(\omega_\gamma)\right)$ are also linearly independent,
generate the same vector space
and, when the bases are ordered by increasing number of cycles, 
the change of basis matrix $M_{S,T}$ is an invertible infinite triangular matrix. 

Choose an ordering of the $T_{L}$ consistent with the partial order on laminations defined above.
Let ${\bf P}=\{P_L(\Phi)\}$ be the Gram-Schmidt orthonormalization of the ${\bf T}=\{T_L(\Phi)\}$ 
with respect to this ordering
and with respect to the inner product $\langle f,g\rangle = \int_{\M} fg \,d\nu.$ 
Let ${\bf A}=(A_{L,L'})$ be the infinite lower-triangular matrix such that ${\bf P}={\bf AT}$. 

Recall the linear operator $\Delta_\Phi$ on the total space $\C^{2|V|}$ of the $\C^2$-bundle on $\G$. 
\begin{theorem}[\cite{Ke1}]\label{qdet}
We have
$$Z(\Phi)=\sqrt{\det(\Delta_\Phi)}.$$
\end{theorem}

One can extract the coefficients of any desired lamination $L$ as follows.
\begin{lemma}\label{cylinders}
For any cylindrical event $E_{L}$, we have 
$$\mu_{nonc}\left(E_{L}\right)=\sum_{L'\ge L} A_{L',L}\int_{\M}\frac{Z(\Phi)}{Z_0}P_{L'}(\Phi)d\nu.$$
\end{lemma}
This sum is finite for any finite graph.

\begin{proof}
The probability of $E_{L}$ is $X_{L}/Z_0$. Write 
$$Z(\Phi)=\sum_{L}X_{L}T_{L}={\bf X}^t{\bf T}(\Phi)={\bf X}^t{\bf A}^{-1}{\bf P}(\Phi)\,.$$
Since ${\bf P}$ is orthonormal, we have $\int_{\M} {\bf P}^t{\bf P}d\nu=\text{Id}$. Hence, 
\begin{eqnarray*}\int_{\M}Z{\bf P}^t{\bf A}d\nu&=&\int_{\M}{\bf X}^t{\bf A}^{-1}{\bf P}(\Phi){\bf P}^t(\Phi){\bf A}d\nu\\
&=&{\bf X}^t{\bf A}^{-1}\left(\int_{\M}{\bf P}(\Phi){\bf P}^t(\Phi)\,d\nu\right){\bf A} ={\bf X}^t\,.
\end{eqnarray*}
Hence, $${\bf X}=\int_{\M}Z{\bf A}^t{\bf P}d\nu\,.$$
Each entry $X_{L}$ is the integral
$$X_{L}=\sum_{L'\ge L} A_{L',L}\int_{\M}Z(\Phi)P_{L'}(\Phi)d\nu.$$
Dividing by $Z_0$ we obtain the result.
\end{proof}

\subsection{Convergence in the flat case}\label{flatconvergence}

We first consider $\Phi$ to be a flat
connection. Associated to this is a measure $\mu^n_{\Phi}$ on noncontractible
CRSFs. Since $\mu_{\Phi}^n$ has a density (independent of the graph) with respect to $\mu^n_{nonc}$, it suffices to show that this latter converges.

The main tool is the following convergence result. Let $x_1,\dots,x_k$ be points of $\Sigma$ and $B_j$  a
small ball around $x_j$. Let $\Phi$ be a flat
connection on $\Sigma\setminus\{B_1\cup\dots\cup B_k\}$.

\begin{theorem}\label{limit}\label{Z/Z}
There exists a function $H\in \mathrm{L^2}(\M^2)$ depending
only on the conformal type of the surface $\Sigma'=\Sigma\setminus\{B_1\cup\dots\cup B_k\}$ such that for any $\Phi'$ not gauge-equivalent to the identity, we have
$$\frac{Z(\Phi)}{Z(\Phi')}\to H(\Phi,\Phi')\,.$$
There exists a bounded
function $F$ on $\M$ depending
only on the conformal type of the surface $\Sigma'=\Sigma\setminus\{B_1\cup\dots\cup B_k\}$ such that 
$$\frac{Z(\Phi)}{Z_0}\to F(\Phi)\,.$$
\end{theorem}
\begin{proof}
The first statement is proved in the same way as Theorem \ref{universality}.

For the second statement, let $Z_0=\sum_{L}X_L$ be the total number of CRSFs.
By the first statement, $\frac{Z_0}{Z_{\Phi}}$ converges for any $\Phi\neq 1$. Now consider its inverse.

We write
\begin{equation}\label{sum}\frac{Z_\Phi}{Z_0}=\sum_{L}\frac{X_L}{Z_0} T_L\,.\end{equation}
First, $T_L\leq e^{O(n(L))}$ because $2-\Tr(\omega)$ is uniformly bounded by a constant over $\M$. Secondly, by Lemma~\ref{decay}, $X_K\le e^{-cn(K)}Z_0$ for arbitrarily small $c>0$. Since there are at most
$n(K)^M$ laminations $L$ with complexity $n(L)=n(K)$, where $M$ is the number of edges of $H$, 
the sum~\eqref{sum} is bounded by a convergent series. 
\end{proof}

\begin{theorem}\label{noncontractible}
Let $\Sigma$ be a compact non-simply connected Riemann surface.
There is a conformally
invariant measure $\P_{nonc}$ on $(\Omega,\B)$ supported on noncontractible multicurves, such that for
any sequence $(\G_n)_{n\ge1}$ of graphs conformally
approximating $\Sigma$,
the measures $\P_{nonc}^{n}$ on noncontractible CRSFs of $\G_n$ converge to $\P_{nonc}$. 
\end{theorem}

The main result of \cite{Ke2} shows that the homotopy classes on $\Sigma$ of the noncontractible loops
have a conformally invariant limit distribution.

\begin{proof}
Take points $z_1,\dots,z_k$ in $\Sigma$, take $\delta>0$ small, and for each $i$
let $B_i$ be the ball of small radius $\delta$ around $z_i$. 
Let $\G^n_B=\G^n\setminus\{B_1\cup\dots\cup B_k\}$ and $\P^n_B$ the associated measure on 
multiloops of noncontractible CRSFs on $\G_n$ whose loops stay in $\G^n_B$.

Consider $L$ a finite lamination in $\Sigma\setminus\{B_1\cup\cdots\cup B_k\}$, which has no peripheral curves
(no curves isotopic to one of the boundary curves $\partial B_i$). It can also be thought of as a lamination of $\G^n_B$.
Let $E_L$ be the event that a CRSF of $\G^n_B$ has lamination~$L$.

Up to errors uniform in $n$ and tending to zero with $\delta$, 
$$\P^n(E_L~|~\text{no peripheral curves}) = \P_B^n\left(E_L~|~\text{no peripheral curves}\right).$$
(Here the term on the left is equal to $\P^n(E_L)$ because the connection is flat.)
This follows from the sampling algorithm
since removing one or more very small disks does not change the distribution of the LERW away from those disks.

We need to show that $\lim_{n\to\infty}\P^n_B(E_L)$ exists and depends only on the conformal
type of the domain $\Sigma\setminus\{B_1\cup\dots\cup B_k\}$.

By Lemma~\ref{cylinders} the probability $\P^n_B(E_L)$ is given by a sum over $L'\ge L$
of integrals over $\M$ of $Z(\Phi)/Z_0$ times a function $P_{L'}(\Phi)$ independent of $n$.
By Corollary~\ref{Z/Z} the integrand $Z(\Phi)/Z_0$ converges and is bounded independently of $n$. Thus, by bounded convergence, for each $L'$ the integral $\langle \frac{Z(\Phi)}{Z_0},P_{L'}\rangle=\int_{\M} \frac{Z(\Phi)}{Z_0}P_{L'}d\nu$ converges.

We need to show that the sum (weighted by the coefficient $A_{L,L'}$) over $L'$ converges. 
We have
\be\label{PLsum}
\langle Z(\Phi),P_{L'}\rangle=\langle\sum_K X_KT_K,P_{L'}\rangle=\langle\sum_{K\ge L'} X_KT_K,P_{L'}\rangle\ee
since $\langle T_K,P_{L'}\rangle=0$ unless $K\ge L'$. We also have $|\langle T_K,P_{L'}\rangle|\le 4^{|K|}=e^{O(n(K))}$ by the Cauchy--Schwarz inequality and because $|K|=\Theta(n(K))$. 
 
Furthermore, by the Gram-Schmidt process, we see that $A_{L,L'}$ is growing at most exponentially in $n(L')$.
That is, there exists $\xi>0$ such that for any $L'$, we have $\vert A_{L,L'}\vert \le \xi^{n(L)}$.

Using the above bounds the sum~\eqref{PLsum} is bounded by $e^{-c'n(L)}Z_0$ for some arbitrarily 
small $c'>0$ and summing over $L'$ and weighting by $A_{L,L'}$ gives a convergent sum.

We now use a classical convergence argument~\cite{Bi}. We showed that the probabilities of any cylindrical event converge. The cylindrical events form a family of sets which is stable under finite intersection. Since it generates the $\sigma$-field
$\B$ it is a determining class, that is, if two probability measures on $\left(\Omega,\B\right)$ coincide on all cylindrical events, then they are equal. 

Since we furthermore have tightness by Theorem~\ref{tightness}, the sequence of probability measures $\P^n$ admits subsequential limits by Prokhorov's theorem. More precisely, for any subsequence, there exists a subsubsequence which converges. Since its value on the cylindrical events is known, there is only one possible limit. Let us call it $\P_{nonc}$. Now, since $\Omega$ is a metric space, this implies that the sequence $\P^n_{nonc}$ converges weakly to $\P_{nonc}$. 
\end{proof}

As a corollary, we obtain the convergence of the measures $\mu^n_{\Phi}$ for any flat connection $\Phi$.

\begin{corollary}\label{flat}
For any flat unitary connection $\Phi$, there exists a probability measure $\P_{\Phi}$ on $(\Omega,\B)$ such that 
$$\P^n_{\Phi}\to\P_{\Phi}$$
in the sense of weak convergence.
\end{corollary}

\begin{proof}
It suffices to show the convergence of this sequence of measures on finite intersections of balls in $\Omega$ 
of small radius because this is a determining class for $\B$. For any curve $\gamma$, there is a small radius $r>0$ such that its tubular $r$-neighborhood retracts onto $\gamma$. Any curve in this $r$-neighborhood and winding once around is isotopic to $\gamma$. On any such neighborhood the density $\prod_{\gamma\subset L}\left(2-\omega_\gamma-\omega_\gamma^{-1}\right)$ is constant, hence the convergence follows by Theorem~\ref{noncontractible}.
\end{proof}

\subsection{Convergence in the curved case (proof of main statement)}\label{curvedconvergence}
The measures $\mu_{LC},\mu_{LC^0}$ converge in the following sense.

\begin{theorem}\label{main}
There exist probability measures $\P_{LC},\P_{LC^0}$ on $(\Omega,\B)$ and $(\Omega^*_1,\B)$ respectively 
such that for any sequence $(\G_n)_{n\ge 1}$ of graphs, geodesically embedded on $\Sigma$ 
and conformally approximating $\Sigma$, the sequences of probability measures $\P_{LC}^{n}$ and $\P_{LC^0}^n$ converge weakly towards 
respectively $\P_{LC}$ and $\P_{LC^0}$. 
\end{theorem}
\begin{proof}
Let us approximate $\Sigma$ by a polygonal surface $\Sigma_{\eps}$, 
that is, with a surface which is flat except for conical singularities.
A standard way to do this is to take a fine triangulation of the surface (with triangles
of diameter at most $\eps$ and whose angles are bounded from below), and replace each triangle
with the Euclidean triangle with the same edge lengths. As $\eps\to0$ the conformal structure
on $\Sigma_\eps$ converges to that of $\Sigma$. (Indeed, there is a homeomorphism $\psi_{\eps}$ from $\Sigma_\eps$
to $\Sigma$ which is $(1+o(1))$-biLipschitz and therefore $(1+o(1))$-quasiconformal.)

Any graph embedded on $\Sigma$ or $\Sigma_\eps$
can be embedded on $\Sigma_{\eps}$ or $\Sigma$ using $\psi_{\eps}^{-1}$ or $\psi_{\eps}$; furthermore a graph conformally close to
$\Sigma_\eps$ will have image on $\Sigma$ conformally close to $\Sigma$, and vice versa.

Let $z_1,\dots,z_k$ be the vertices of $\Sigma_{\eps}$. The Levi-Civita connection on $\Sigma_{\eps}$
is a flat connection on $\Sigma_{\eps}\setminus\{z_1,\dots,z_k\}$ and approximates the 
Levi-Civita connection on $\Sigma$,
in the sense that for small $\eps$ the curvature enclosed by any loop (chosen independently of the triangulation) 
is close for both connections.
Restricting to $\G_n$, this shows that $\P^n_{LC,\eps}$ is close to $\P^n_{LC}$ (since cylinder events have 
measures which are within $o(1)$ of each other). 

By Corollary \ref{flat}, if we fix $\Sigma_{\eps}$ and take $\G_n$ embedded on $\Sigma_{\eps}$ and
conformally approximating it as $n\to\infty$,
the measures $\P^n_{LC,\eps}$ converge as $n\to\infty$ to a limit  $\P_{LC,\eps}$. 
Similarly for $\P^n_{LC^0,\eps}$. 

For any $\varepsilon>0$ the probabilities of the cylindrical events (away from the singularities) are 
determined by the function $F$ of Lemma~\ref{limit}. These are expressed as integrals of Green's function on paths avoiding the singularities. The addition of a singularity modifies the Green's function, and hence the probability, 
by a negligible amount. 

By a diagonal argument, we can take $n\to\infty$ and then $\eps\to 0$ and conclude
that~$\P^{n}_{LC}$ and~$\P^n_{LC^0}$ converge weakly as $n\to\infty$. 
\end{proof}

From experimental simulations, the probability of getting two or more loops in~$\mu_{LC}$ on the round sphere is on the order of one percent. Hence measures~$\mu_{LC}$ and~$\mu_{LC^0}$ are close for the total variation distance.

The measure $\P_{LC^0}$ is a limit of~$\P_{LC}$ when the metric is scaled by a factor $t\to 0$. 
Hence it is not obvious that Theorem~\ref{nomicroloops} implies that~$\P_{LC^0}$ also 
is supported on macroscopic loops. However, we conjecture it to be true.

\begin{conjecture}
$\P_{LC^0}$ is supported on $\Omega_1$ that is 
$$\P_{LC^0}\left(\text{the area of the loop is zero}\right)=0\,.$$
\end{conjecture}

\subsection{A comment and two applications}\label{commentandapplications}

The convergence result of probability measures $\mu_{c,\alpha}$ on CRSFs is actually more general and can be adapted for a wide range of functions $\alpha$ on the cycles, not necessarily coming from connections. This is due to the fact that the crucial convergence argument is made for the uniform measure on noncontractible CRSFs (Theorem~\ref{noncontractible}). We now state two applications of the previous theorems. 

As an application of Theorem~\ref{main}, we obtain a result on higher moments of the area of a uniform CRST, also mentioned in the paper~\cite{KKW}. 
Let $A$ denote the combinatorial area (number of faces) of the cycle of a cycle-rooted spanning tree of $\G_n$. Let $\theta=A/n^2$ be the Euclidean area of the cycle. We denote by $\E_{\text{unif}}^n$ the expectation with respect to the uniform measure on CRSTs on $\G_n$.

\begin{corollary}\label{ak}
For $k\ge 2$ there exists $a_k(D)>0$ such that 
$$\E_{\text{unif}}^n\left(A^k\right)=a_k(D)n^{2k-2}(1+o(1))\,.$$
\end{corollary}
\begin{proof}Let $k\geq 2$. We have
\begin{eqnarray*}
\E_{\text{unif}}^n\left(A^k\right)&=&\E^n_{LC^0}\left(A^{k-2}\right)\E_{\text{unif}}^n\left(A^{2}\right)\\
&=& n^{2k-4}\E_{LC^0}^n\left(\theta^{k-2}\right)\E_{\text{unif}}^n\left(A^{2}\right)\\
&=& n^{2k-2}C(D)|D|\E_{LC^0}\left(\theta^{k-2}\right)(1+o(1))\,,
\end{eqnarray*}
where the last equality follows from Theorem~$6$ of~\cite{KKW} (here $C(D)$ is a constant proportional to the torsional rigidity of the domain) and the weak convergence of $\P_{LC^0}^n$ to $\P_{LC^0}$. Since $\theta$ is bounded and for $\P_{LC^0}$ is with positive probability nonzero by Theorem~\ref{nonzero}, 
the limit $\E_{LC^0}\left(\theta^{k-2}\right)$ is a positive real. The corollary is proved by taking $a_k(D)=C(D)|D|\E_{LC^0}\left(\theta^{k-2}\right)$.
\end{proof}

As an application of Theorem~\ref{noncontractible}, let us state the following corollary.

\begin{corollary}
Consider a uniform spanning forest on an annulus-graph wired on its boundary. Then the simple closed
curve separating the two connected components has a conformally invariant limit which is given by the measure $\P_{nonc}$ conditional on having only one component. 
\end{corollary}
\begin{proof}
This follows from the fact that the dual of a wired essential forest on the annulus is a uniform noncontractible CRST on the annulus with free boundary conditions. The measure is thus given by $\mu^n_{nonc}$ conditional on having one loop. The convergence follows from Theorem~\ref{noncontractible}.
\end{proof} 

Figure~\ref{wiredtree} shows a sample of this interface between the two tree components of a 
wired uniform spanning forest on the annulus.

\begin{figure}[ht]
\centering
\includegraphics[width=10cm]{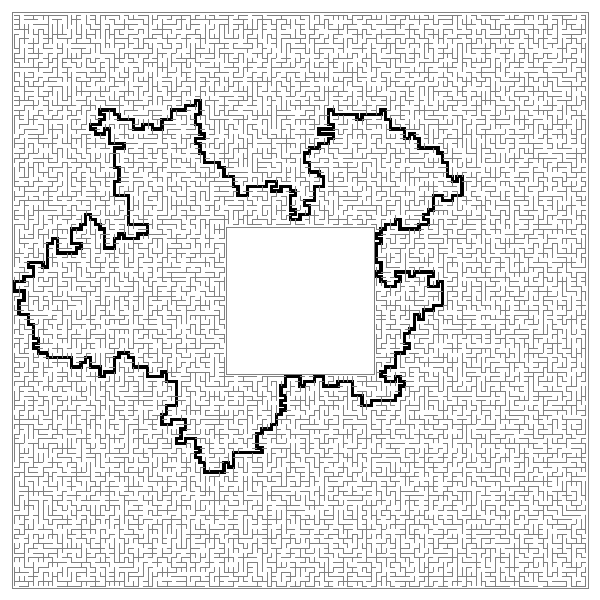}
\caption{A uniform noncontractible CRST on an annulus in the square grid}\label{wiredtree}
\end{figure}

\section{Properties of the measures}\label{properties}

In this section, we mention a few properties of the measures on CRSFs on surfaces we have been considering. 

\subsection{Markov property}\label{Markov}

CRSFs on surfaces with general cycle weights satisfy the following \emph{spatial Markov property}.
Consider a graph $\G$ embedded in a compact oriented surface $\Sigma$, possibly with boundary.  
Let $\alpha:\Omega_1\to\R_{>0}$ be any positive weight function on the cycles of~$\G$.

Let $\ga$ be the random CRSF on $\G$. Let $\{\gamma\}=\{\gamma_1,\ldots,\gamma_k\}$ be a family of its cycles. These cycles 
separate $\Sigma$ in a number of connected components $\Sigma_1,\ldots,\Sigma_r$. 
For $i=1,\ldots,r$, denote by $\G_i$ the intersection of $\G$ and the closure of $\Sigma_i$ (\emph{i.e.} $\Sigma_i$ along with its boundary) and by $\partial \G_i$ the boundary cycles. 

\begin{prop}
Conditional on $\{\gamma\}\subset\ga$, the random CRSF $\ga$ is equal in distribution to $$\{\gamma\}\sqcup_{i=1}^{r}\ga_i\,.$$
\end{prop}

\begin{proof} 
A CRSF $\Gamma$ on $\G$ which contains $\{\gamma\}$ is the union $\{\gamma\}\sqcup_{i=1}^{r}\Gamma_i$ of $\{\gamma\}$ with essential CRSFs $\Gamma_i$ on each one of the connected component $\Sigma_i$ of $\Sigma\setminus\{\gamma\}$. The proposition then follows directly from the cycle popping algorithm. 
\end{proof}

\subsection{Restriction property}

Let $D_1\subset D$ be two planar Jordan domains. Let $\G$ be a finite graph approximation of $D$. Let $\Phi$ be a connection on a line bundle over $\G$. We denote $\det \Delta^D(\Phi)$ the line bundle Laplacian on $\G$ with connection $\Phi$. For any subset $S$ of the set of vertices of $\G$, we denote by $\det \Delta^{D}_{S}(\Phi)$ the line bundle Laplacian with Dirichlet boundary conditions on $S$, see~\cite{Ke1}. 

The measure on multicurves that stay in $D_1$ is absolutely continuous with respect to the measure on $D$. The Radon-Nikodym derivative is given by a cross-ratio of determinants of the Laplacian with different boundary conditions:

\begin{lemma}For any finite set of simple non-intersecting curves $\{\gamma\}\subset D_1$, we have
$$\frac{\mu_{\Phi}^{D_1}\left(\{\gamma\}\right)}{\mu_{\Phi}^D\left(\{\gamma\}\right)}=\left(\frac{\det \Delta^{D_1}(\Phi)}{\det \Delta^{D_1}_{\{\gamma\}}(\Phi)}\right)/\left(\frac{\det \Delta^D(\Phi)}{\det \Delta^D_{\{\gamma\}}(\Phi)}\right)\,.$$
\end{lemma}
\begin{proof}
The proof follows by direct computation using the Forman-Kenyon matrix tree theorem.
\end{proof}

\subsection{Stochastic domination}
Recall from Section~\ref{algorithm} the notation $\mu_{c,\alpha}$ for measures that assign a CRSF $\Gamma$ a probability proportional to $\prod_{e\in\Gamma}c(e)\prod_{\gamma\subset\Gamma}\alpha(\gamma)$, where each $\alpha(\gamma)\in[0,1]$.

\begin{lemma}\label{negcorrelations}
Let $S_1\subset S_2$ be two subgraphs of $\G$. Let $\P_1$ and $\P_2$ be essential CRSF measures $\mu_{c,\alpha}$  
with Dirichlet boundary conditions on $S_1$ and $S_2$, respectively. For any curve $\gamma\in\G\setminus S_2$, we have
$$\P_1(\gamma)\geq \P_2(\gamma)\,.$$
\end{lemma}
\begin{proof}
This follows from the cycle-popping algorithm, as follows. We couple the cycle-popping algorithms
for  $\G_1$ and $\G_2$, starting with identical stacks of cards at each vertex.
When a cycle is found for $\G_1$ which is also a cycle for $\G_2$, keep them both or discard them both according to the coin toss,
increasing $S_1,S_2$ appropriately.
If a cycle is found for $\G_1$ which is not a cycle of $\G_2$ (that is, defines a LERW connected to the current $S_2$,
keep that LERW in $\G_2$ (and, in addition, every part of the cycle which connects to the current boundary $S_2$) and toss a coin to determine whether or not to keep the cycle in $\G_1$. For either outcome of the coin toss it
is still true that $S_2$, the current union of boundaries of $\G_2$, contains $S_1$, the current union of boundaries of $\G_1$. 
Continue until cycle $\gamma$ is found for $\G_1$; at that point it will be kept in $\G_2$ if and only if it does not intersect the current $S_2$.
\end{proof}

Note that in the case the measures come from a line bundle connection $\Phi$, this implies
$$\frac{\det \Delta_{S_1\cup\gamma}(\Phi)}{\det \Delta_{S_1}(\Phi)}\geq\frac{\det \Delta_{S_2\cup\gamma}(\Phi)}{\det \Delta_{S_2}(\Phi)}\,,$$
which is a non trivial potential theoretic consideration (which can be translated in terms of Dirichlet-to-Neumann map).

\section{Questions}\label{questions}

\begin{enumerate}

\item
Can the measures $\P_{nonc}$, $\P_{LC^0}$, and $\P_{LC}$ be defined directly instead of via limits of CRSF measures?
For example via a stochastic differential equation, like a variant of $\mathrm{SLE}_2$ defined on Riemann surfaces subject to some potential depending on the metric? 

\item
On the round sphere for the measure $\mu_{LC^0}$, can one use the connection Laplacian to
say more about the shape of the cycle, as is done in \cite{KKW} in the flat case?

\item What is the function $x\mapsto\P_{LC^0}\left(x \,\text{is enclosed by the loop}\right)$ for the unit disk?

\item
What can be said about the Gaussian Free Field associated to the line bundle Laplacian? Are our 
loop models related to this GFF? In particular, for a choice of an infinite curvature, we expect to obtain loops at all scales. 

\item
What is the right framework for the study of $\P_{LC}$? What is the distribution of the number of loops?

\item 
Is there a probabilistic interpretation for the coefficients of the triangular matrix $\mathbf{A}$ defined in Section~\ref{integrals}?

\item 
What is the scaling limit of CRSFs in higher dimension?

\item 
What can be said about the scaling limit of measures $\mu_\alpha$ on CRSFs for weight functions $\alpha$ not coming from a connection?

\item Can our loop measures be used to study the scaling limit of waves of avalanches in the sandpile model in the scaling limit? See \cite{IKP} for a definition of waves in this model.

\end{enumerate}


\begin{thebibliography}{9999999}

\bibitem[Ald83]{Al}
\textsc{D.\ Aldous}, 
\textit{On the time taken by random walks on finite groups to visit every state}, Wahrsch. Verw. Gebeite 62 (1983), 361--374.
\href{http://www.ams.org/mathscinet-getitem?mr=688644}{MR0688644 (84i:60013)}

\bibitem[AB99]{AB}
\textsc{M.\ Aizenman, A.\ Burchard},
\textit{H\"older regularity and dimension bounds for random curves},
Duke Math. J. 99 (3) (1999), 419--453. 
\href{http://www.ams.org/mathscinet-getitem?mr=1712629}{MR1712629 (2000i:60012)}

\bibitem[AB+99]{ABNW}
\textsc{M.\ Aizenman, A.\ Burchard, C.\ M.\ Newman, D.\ B.\ Wilson},
\textit{Scaling limits for minimal and random spanning trees in two dimensions},
Random Structures and Algorithms 15 (3--4) (1999), 319â-367.
\href{http://www.ams.org/mathscinet-getitem?mr=1716768}{MR1716768 (2001c:60151)}

\bibitem[BD14]{BD}
\textsc{S. Benoist, J. Dub\'{e}dat},
\textit{An $\text{SLE}_2$ loop measure}, to appear in Ann. Inst. Henri Poincar\'e Probab. Stat. \href{http://arxiv.org/abs/1405.7880v1}{arXiv:1405.7880}

\bibitem[Big97]{Big}
\textsc{N.\ Biggs},
\textit{Algebraic potential theory on graphs}, Bull. London Math. Soc. 29 (1997), no. 6, 641--682.
\href{http://www.ams.org/mathscinet-getitem?mr=1468054}{MR1468054 (98m:05120)}

\bibitem[Bil99]{Bi}
\textsc{P.\ Billingsley},
\textit{Convergence of probability measures}, second edition, Wiley Series in Probability and Statistics: Probability and Statistics, Wiley \& Sons, New York, 1999.
\href{http://www.ams.org/mathscinet-getitem?mr=1700749}{MR1700749 (2000e:60008)}

\bibitem[BP93]{BP}
\textsc{R.\ Burton, R.\ Pemantle},
\textit{Local characteristics, entropy, and limit theorems for spanning trees and domino tilings via transfer impedances}, Ann. Prob. 21 (3) (1993), 1329--1371.
\href{http://www.ams.org/mathscinet-getitem?mr=1235419}{MR1235419 (94m:60019)}

\bibitem[FG06]{FG}
\textsc{V.\ Fock, A.\ Goncharov},
\textit{Moduli spaces of local systems and higher Teichm\"{u}ller theory}, Publ. Math. Inst. Hautes \'Etudes Sci. No. 103 (2006), 1--211.
\href{http://www.ams.org/mathscinet-getitem?mr=2233852}{MR2233852 (2009k:32011)}

\bibitem[For93]{Fo}
\textsc{R.\ Forman},
\textit{Determinants of Laplacians on graphs}, Topology 32 (1) (1993), 35--46. 
\href{http://www.ams.org/mathscinet-getitem?mr=1204404}{MR1204404 (94g:58247)}

\bibitem[GK06]{GK}
\textsc{E.\ Gin\'e, V.\ Koltchinskii},
\textit{Empirical graph Laplacian approximation of Laplace--Beltrami operators: Large sample results}, IMS Lecture Notes Monograph Series 2006, Vol. 51, 238--259.
\href{http://www.ams.org/mathscinet-getitem?mr=2387773}{MR2387773 (2009a:58051)}

\bibitem[HK+06]{HKPV}
\textsc{J.\ B.\ Hough, M.\ Krishnapur, Y.\ Peres, B.\ Vir\'ag},
\textit{Determinantal processes and independence}, Prob. Surveys (3) (2006), 206--229.
\href{http://www.ams.org/mathscinet-getitem?mr=2216966}{MR2216966 (2006m:60068)}


\bibitem[It\^{o}62]{Ito}
\textsc{K.\ It\^{o}},
\textit{The Brownian motion and tensor fields on Riemannian manifold}, In: Proc. Intern. Congr. Mathemat. (Stockholm), 536--539.
\href{http://www.ams.org/mathscinet-getitem?mr=176500}{MR0176500 (31 \#772)}

\bibitem[IKP94]{IKP} \textsc{E.\ V.\ Ivashkevich, D.\ V.\ Ktitarev, V.\ B.\
Priezzhev}, 
\textit{Waves of topplings in an Abelian sandpile}, Physica A 209
(1994), 347--360.


\bibitem[KW13]{KasWu}
\textsc{A.\ Kassel, W.\ Wu},
\textit{Transfer current and pattern fields in spanning trees}, (2013), to appear in Probab. Theory Related Fields. \href{http://arxiv.org/abs/1312.2946}{arXiv:1312.2946} 

\bibitem[KKW13]{KKW}
\textsc{A.\ Kassel, R.\ Kenyon, W.\ Wu},
\textit{Random two-component spanning forests}, (2013), to appear in Ann. Inst. Henri Poincar\'e Probab. Stat. \href{http://arxiv.org/abs/1203.4858}{arXiv:1203.4858}

\bibitem[Ken11]{Ke1}
\textsc{R.\ W.\ Kenyon},
\textit{Spanning forests and the vector bundle Laplacian}, Ann. Prob. 39 (5) (2011), 1983--2017.
\href{http://www.ams.org/mathscinet-getitem?mr=2884879}{MR2884879 (2012k:82011)}

\bibitem[Ken11b]{Ke2}
\textsc{R.\ W.\ Kenyon},
\textit{Conformal invariance of loops in the double dimer model}, Comm. Math. Phys. 326 (2014), no. 2, 477--497.
\href{http://www.ams.org/mathscinet-getitem?mr=3165463}{MR3165463}

\bibitem[KPW00]{KPW} 
\textsc{R. Kenyon, J. Propp, D. B. Wilson},
{\textit Trees and matchings}. Electron. J. Combin. 7 (2000), Research Paper 25, 34 pp. (electronic)

\bibitem[LF07]{LF}
\textsc{G.\ Lawler, J.\ A.\ Trujillo Ferreras},
\textit{Random walk loop soup}, Trans. Amer. Math. Soc. 359 (2007), no. 2, 767--787 (electronic). \href{http://www.ams.org/mathscinet-getitem?mr=2255196}{MR 2255196 (2008k:60084)}

\bibitem[LSW04]{LSW}
\textsc{G.\ Lawler, O.\ Schramm, W.\ Werner}, 
\textit{Conformal invariance of planar loop-erased random walks and uniform spanning trees}, 
Ann. Probab. 32 (2004), no. 1B, 939--995.
\href{http://www.ams.org/mathscinet-getitem?mr=2044671}{MR2044671 (2005f:82043)}

\bibitem[LS77]{LS}
\textsc{R.\ C.\ Lyndon, P.\ E.\ Schupp},
\textit{Combinatorial group theory}, reprint of the 1977 edition, Classics in Mathematics, Springer-Verlag, Berlin, 2001.
\href{http://www.ams.org/mathscinet-getitem?mr=1812024}{MR1812024 (2001i:20064)}

\bibitem[Sch00]{Sc}
\textsc{O.\ Schramm},
\textit{Scaling limits of loop-erased random walks and uniform spanning trees},
Israel J. Math. 118 (2000), 221--288.
\href{http://www.ams.org/mathscinet-getitem?mr=1776084}{MR1776084 (2001m:60227)}

\bibitem[Weh62]{Wehn}
\textsc{D. Wehn},
\textit{Probabilities on Lie groups}, Proc. Nat. Acad. Sci. U.S.A. 48 (1962) p. 791Ð795.

\bibitem[Wil96]{Wi}
\textsc{D.\ B.\ Wilson},
\textit{Generating random spanning trees more quickly than the cover time}, Proceedings of the Twenty-eigth Annual ACM Symposium on the Theory of Computing (Philadelphia, PA, 1996), 296--303, ACM, New York, 1996.
\href{http://www.ams.org/mathscinet-getitem?mr=1427525}{MR1427525}
\end{thebibliography}
\end{document}